\documentclass[11pt,reqno]{amsart}

\usepackage{amsthm}
\usepackage{amssymb}
\usepackage{graphicx}
\usepackage{latexsym}
\usepackage{multicol}
\usepackage{verbatim,enumerate}
\usepackage{accents}
\usepackage[usenames]{color}
\usepackage[colorlinks=true, linkcolor=blue, citecolor=blue, urlcolor=blue]{hyperref}
\usepackage{hyperref}
\usepackage{amsmath, amscd}

\advance\textwidth by 1.3in \advance\oddsidemargin by -.6in \advance\evensidemargin by -.6in
\def\Gaff{\widehat{\mathfrak g}} 
\def\G{\mathfrak g} 
\def\H{\mathfrak h} 
\def\Haff{\widehat{\mathfrak h}} 
\def\Baff{\widehat{\mathfrak b}} 
\def\CG{\mathfrak{Cg}} 
\def\CN{\mathfrak{Cn}} 
\def\CH{\mathfrak{Ch}} 

\DeclareMathOperator{\D}{D}

\DeclareMathOperator{\KRR}{KR}
\DeclareMathOperator{\htt}{ht}
\def\a{\alpha}
\def\l{\lambda}
\def\d{\delta}

\theoremstyle{definition}
\newtheorem*{cor}{Corollary}
\newtheorem*{lem}{Lemma}
\newtheorem*{prop}{Proposition}

\newtheorem{thm}{Theorem}
\newtheorem*{theom}{Theorem}

\theoremstyle{definition}
\newtheorem*{defn}{Definition}

\theoremstyle{definition}

\newtheorem*{rem}{Remark}

\newenvironment{pf}{\proof}{\endproof}
\newcounter{cnt}
\newenvironment{enumerit}{\begin{list}{{\hfill\rm(\roman{cnt})\hfill}}{%
\settowidth{\labelwidth}{{\rm(iv)}}\leftmargin=\labelwidth%
\advance\leftmargin by \labelsep\rightmargin=0pt\usecounter{cnt}}}{\end{list}} \makeatletter
\def\mydggeometry{\makeatletter\dg@YGRID=1\dg@XGRID=20\unitlength=0.003pt\makeatother}
\makeatother \theoremstyle{remark}

\numberwithin{equation}{section}
\let\bwdg\bigwedge
\def\bigwedge{{\textstyle\bwdg}}

\newcommand{\charc}{\hbox{\rm ch}\,}

\newcommand{\wt}{\operatorname{wt}}

\newcommand{\nc}{\newcommand}
\newcommand{\rnc}{\renewcommand}

\nc{\cal}{\mathcal} \nc{\goth}{\mathfrak} \rnc{\bold}{\mathbf}

\newcommand{\supp}{\operatorname{supp}}

\nc\bomega{{\mbox{\boldmath $\omega$}}} \nc\bpsi{{\mbox{\boldmath $\Psi$}}}
 \nc\balpha{{\mbox{\boldmath $\alpha$}}}
 \nc\bpi{{\mbox{\boldmath $\pi$}}}
 \nc\bvpi{{\mbox{\boldmath $\varpi$}}}
\nc\chara{\operatorname{ch}}

  \nc\bxi{{\mbox{\boldmath $\xi$}}}
\nc\bmu{{\mbox{\boldmath $\mu$}}} \nc\bcN{{\mbox{\boldmath $\cal{N}$}}} \nc\bcm{{\mbox{\boldmath $\cal{M}$}}} \nc\blambda{{\mbox{\boldmath
$\lambda$}}}\nc\bnu{{\mbox{\boldmath $\nu$}}}

\newcommand{\lie}[1]{\mathfrak{#1}}

\makeatletter
\def\section{\def\@secnumfont{\mdseries}\@startsection{section}{1}%
  \z@{.7\linespacing\@plus\linespacing}{.5\linespacing}%
  {\normalfont\scshape\centering}}
\def\subsection{\def\@secnumfont{\bfseries}\@startsection{subsection}{2}%
  {\parindent}{.5\linespacing\@plus.7\linespacing}{-.5em}%
  {\normalfont\bfseries}}
\makeatother

 \nc{\Hom}{\operatorname{Hom}}
  \nc{\mode}{\operatorname{mod}}
\nc{\End}{\operatorname{End}} \nc{\wh}[1]{\widehat{#1}} \nc{\Ext}{\operatorname{Ext}} \nc{\ch}{\text{ch}} \nc{\ev}{\operatorname{ev}}
\nc{\Ob}{\operatorname{Ob}} \nc{\soc}{\operatorname{soc}} \nc{\rad}{\operatorname{rad}} \nc{\head}{\operatorname{head}}

\def\gr{\operatorname{gr}}
\def\mult{\operatorname{mult}}

\def\loc{\operatorname{loc}}

 \nc{\Cal}{\cal} \nc{\Xp}[1]{X^+(#1)} \nc{\Xm}[1]{X^-(#1)}
\nc{\on}{\operatorname} \nc{\Z}{{\bold Z}} \nc{\J}{{\cal J}}  \nc{\Q}{{\bold Q}}

\nc{\N}{{\bold N}}  \nc\boa{\bold a} \nc\bob{\bold b} \nc\boc{\bold c} \nc\bod{\bold d} \nc\boe{\bold e} \nc\bof{\bold f} \nc\bog{\bold g}
\nc\boh{\bold h} \nc\boi{\bold i} \nc\boj{\bold j} \nc\bok{\bold k} \nc\bol{\bold l} \nc\bom{\bold m} \nc\bon{\mathbb n} \nc\boo{\bold o}
\nc\bop{\bold p} \nc\boq{\bold q} \nc\bor{\bold r} \nc\bos{\bold s} \nc\boT{\bold t} \nc\boF{\bold F} \nc\bou{\bold u} \nc\bov{\bold v}
\nc\bow{\bold w} \nc\boz{\bold z}\nc\ba{\bold A} \nc\bb{\bold B} \nc\bc{\mathbb C} \nc\bd{\bold D} \nc\be{\bold E} \nc\bg{\bold
G} \nc\bh{\bold H} \nc\bi{\bold I} \nc\bj{\bold J} \nc\bk{\bold K} \nc\bl{\bold L} \nc\bm{\bold M} \nc\bn{\mathbb N} \nc\bo{\bold O} \nc\bp{\bold
P} \nc\bq{\bold Q} \nc\br{\bold R} \nc\bs{\bold S} \nc\bt{\bold T} \nc\bu{\bold U} \nc\bv{\bold V} \nc\bw{\bold W} \nc\bz{\mathbb Z} \nc\bx{\bold
x} \nc\KR{\bold{KR}} \nc\rk{\bold{rk}} \nc\het{\text{ht }}

\nc\toa{\tilde a} \nc\tob{\tilde b} \nc\toc{\tilde c} \nc\tod{\tilde d} \nc\toe{\tilde e} \nc\tof{\tilde f} \nc\tog{\tilde g} \nc\toh{\tilde h}
\nc\toi{\tilde i} \nc\toj{\tilde j} \nc\tok{\tilde k} \nc\tol{\tilde l} \nc\tom{\tilde m} \nc\ton{\tilde n} \nc\too{\tilde o} \nc\toq{\tilde q}
\nc\tor{\tilde r} \nc\tos{\tilde s} \nc\toT{\tilde t} \nc\tou{\tilde u} \nc\tov{\tilde v} \nc\tow{\tilde w} \nc\toz{\tilde z} \nc\woi{w_{\omega_i}}
\begin{document}


\title{Twisted Demazure modules, fusion product decomposition and twisted Q--systems}

\author{Deniz Kus}
\address{Mathematisches Institut, Universit\" at zu K\" oln, Germany}
\email{dkus@math.uni-koeln.de}
\thanks{D.K. was partially supported by the “SFB/TR 12-Symmetries and
Universality in Mesoscopic Systems”.}
\author{R. Venkatesh}
\thanks{}
\address{Tata Institute of Fundamental Research, Mumbai, India}
\email{r.venkatmaths@gmail.com}

\begin{abstract}
In this paper, we introduce a family of indecomposable finite--dimensional graded modules for the twisted current algebras. These modules are indexed by an $|R^+|$--tuple of partitions $\bxi=(\xi^{\alpha})_{\alpha\in R^+}$ satisfying a natural compatibility condition. We give three equivalent presentations of these modules and show that for a particular choice of $\bxi$ these modules become isomorphic to Demazure modules in various levels for the twisted affine algebras. As a consequence we see that the defining relations of twisted Demazure modules can be greatly simplified. Furthermore, we investigate the notion of fusion products for twisted modules, first defined in \cite{FL99} for untwisted modules, and use the simplified presentation to prove a fusion product decomposition of twisted Demazure modules. As a consequence we prove that twisted Demazure modules can be obtained by taking the associated graded modules of (untwisted) Demazure modules for simply--laced affine algebras. Furthermore we give a semi--infinite fusion product construction for the irreducible representations of twisted affine algebras. Finally, we prove that the twisted $Q$--sytem defined in \cite{HKOTT02} extends to a non-canonical short exact sequence of fusion products of twisted Demazure modules.
\end{abstract}
\maketitle

\section*{Introduction}
The twisted quantum affine algebras and their representations have been intensively studied. For instance the finite--dimensional irreducible representations are classified in \cite{CP98} in terms of their highest weights. However, the structure of these representations is still unknown except in certain special cases. A certain infinite class of irreducible finite--dimensional representations are called the Kirillov--Reshetikhin modules. Many conjectures for the characters of Kirillov-Reshetikhin modules and of their tensor products have been formulated in \cite{HKOTT02,KS95,R87} and a conjectural induction rule called the twisted $Q$--system is given in \cite{HKOTT02}.
These conjectures are formulated before for the untwisted cases in \cite{K87,KR87} by observing the Bethe Ansatz related to solvable lattice models.
There are many results for these conjectures and related problems; for untwisted quantum affine algebras we refer to \cite{Na03,HKOTY99,H06} and for twisted quantum affine algebras see \cite{H10}. 
It has been shown in these papers that the solutions to the $Q$--systems come from a family of short exact sequences of tensor products of suitable Kirillov--Reshetikhin modules. One of the motivations of this paper is to have a better understanding of these short exact sequences.\par
A different approach to this problem is provided in \cite{CM06} and \cite{CM07} respectively for the twisted cases. The goal of these papers was to understand the $q\mapsto 1$ limit of the solutions of the twisted $Q$--system. It is shown in \cite{CM06,CM07} that the solutions are characters of certain finite--dimensional indecomposable graded representations of the twisted current algebras; also called the (twisted) Kirillov--Reshetikhin modules. The interest in the category of finite--dimensional graded representations of twisted current algebras is therefore naturally originated in the context of the representation theory of twisted quantum affine algebras.\par
A different family of indecomposable finite--dimensional graded modules for the twisted current algebras are called the (twisted) Demazure modules $\D(\ell,\lambda)$ and are indexed by pairs $(\ell,\lambda)$, where $\ell$ is a positive integer and $\lambda$ is a dominant integral weight for the underlying simple Lie algebra. A close relationship between level one twisted Demazure modules and local Weyl modules is developed in \cite{CIK14} and \cite{FK11} respectively. Moreover, it can be observed that the classical decomposition of certain twisted Demazure modules determined in \cite{FoL06} coincides with the classical decomposition of certain twisted Kirillov--Reshetikhin modules determined in \cite{CM06,CM07}. Indeed, we observe that any twisted Kirillov--Reshetikhin module is as a module for the twisted current algebra isomorphic to a certain twisted Demazure module. Motivated by this fact, we study the general theory of twisted Demazure modules in positive level representations of twisted affine algebras.
We remark that an isomorphism between certain Kirillov--Reshetikhin modules and certain Demazure modules for the untwisted cases was established before in \cite{CM06} and \cite{FoL07}.\par

For a twisted affine algebra $\widehat{\lie g}$ we denote by $\CG$ the twisted current algebra associated to $\widehat{\lie g}$, which is essentially defined as the special maximal parabolic subalgebra of $\widehat{\lie g}$. Apart from $\widehat{\lie g}$ of type $\tt A_{2n}^{(2)}$ we can realize $\CG$ as the fixed point subalgebra of $\overline{\lie g}\otimes \bc[t]$ under an automorphism induced from a non--trivial diagram automorphism of $\overline{\lie g}$, where $\overline{\lie g}$ is as in Section~\ref{section4}.
For type $\tt A_{2n}^{(2)}$ there are two conjugacy classes of special maximal parabolic subalgebras, where one is realized in the same fashion and one of them has distinguished properties and is called the hyperspecial twisted current algebra (see \cite{CIK14}). The focus of this paper is on the hyperspecial case and following \cite{CIK14} we refer to the remaining twisted algebras as the special twisted current algebras.\par

The study of twisted Demazure modules will proceed by considering two cases for the following reason. The investigation of the special twisted current algebras eventually relies on the understanding of the representation theory of the current algebra $\mathfrak{sl}_2\otimes \bc[t]$, that is the only rank one current subalgebra that can appear. On the other hand the study of hyperspecial twisted current algebras is quite more challenging and one has to deal with the new phenomenon that the rank one twisted current algebra $\tt A_{2}^{(2)}$ appears as a subalgebra. Furthermore, the hyperspecial twisted current algebra is not realized in the same fashion as the special twisted current algebras and hence many technical difficulties show up.\par
Let us describe our results for the hyperspecial twisted current algebra.
For an $|R^+|$--tuple of partitions $\bxi=(\xi^{\alpha})_{\alpha\in R^+}$ we introduce a family of indecomposable finite--dimensional graded modules $V(\bxi)$ and give three equivalent presentations of these modules. We show that the presentation of these modules can be greatly simplified for so--called special fat hook partitions. Our main results
are the following; we refer to Section~\ref{section2}, Section~\ref{section3} and Section~\ref{section6} for the precise definition of the ingredients.
\begin{theom}
For any pair $(\ell,\lambda)\in \bn\times P^+$, there exists a special fat hook partition $\bxi(\ell,\lambda)$, such that we have an isomorphism of $\CG$--modules
$$V(\bxi(\ell,\lambda))\cong \D(\ell,\lambda).$$
\end{theom}
As a consequence we see that the defining relations of twisted Demazure modules given by Mathieu in \cite{M88} can be greatly simplified.\par
Feigin and Loktev introduced the notion of a fusion product of graded representations of the current algebra \cite{FL99}. It was later proved in \cite{CSVW14} and \cite{FoL07} that a Demazure module is a fusion product of ``smaller" Demazure modules. The main problem of defining fusion products for twisted current algebras is that $\CG$ is not stabilized by the Lie algebra homomorphism $\overline{\lie g}[t]\longrightarrow \overline{\lie g}[t], t^k\mapsto (t+a)^k, a\in \bc^{\times}$. This is a major obstacle to generalizing such techniques to the setting of twisted current algebras. For these reasons, new techniques are needed and our approach to this problem is to use untwisted modules and the graded Lie algebra structure on $\CG$ induced by the derivation $d$. Again for the precise definition of the ingredients and a slightly more general formulation see Theorem~\ref{mainthmsection6}.
\begin{theom}
Let $\lambda=\ell\lambda_1+\cdots+\ell\lambda_p+\lambda_{0}$ be an arbitrary decomposition, where $\lambda_k\in P^+$ for $0\leq k \leq p$. Then we have an isomorphism of $\CG$--modules 
$$V(\bxi(\ell,\lambda))\cong \D_{\overline{\lie g}}(\ell,\ell\lambda_1)*\cdots* \D_{\overline{\lie g}}(\ell,\ell\lambda_{p})*V(\bxi(\ell,\lambda_{0}))$$
\end{theom}
As a corollary of the previous theorem we obtain that certain twisted Demazure modules can be obtained by taking the associated graded modules of untwisted Demazure modules, which was previously known for local Weyl modules for special twisted current algebras \cite{FK11}.
Another application of our theorem is the following semi--infinite fusion product construction:
\begin{theom}
Let $\lambda$ be a dominant integral $\lie g$--weight such that $\Lambda=\ell \Lambda_0+\lambda$ is a dominant integral $\widehat{\lie g}$--weight. Furthermore, let $\bold{V}^\infty_{\ell, \lambda}$ be the direct limit of 
$$\ev_0^{*}V(\lambda)\hookrightarrow \D_{\overline{\lie g}}(\ell, \ell\theta)*\ev_0^{*}V(\lambda)\hookrightarrow \D_{\overline{\lie g}}(\ell, \ell\theta)*\D_{\overline{\lie g}}(\ell, \ell\theta)*\ev_0^{*}V(\lambda)\hookrightarrow \cdots$$
Then $\widehat{V}(\Lambda) \ \text{and} \ \bold{V}^\infty_{\ell, \lambda}$ are isomorphic as $\CG$--modules.
\end{theom}
Finally, we prove the following; for details we refer to Section~\ref{section7}.
\begin{theom}
The twisted $Q$--system extends to a non-canonical short exact sequence of fusion products of twisted Demazure modules.

\end{theom}

Our paper is organized as follows. Section~\ref{section1} establishes the basic notation and elementary
results needed in the rest of the paper. In Section~\ref{section2}, we define the modules $V(\bxi)$ for the hyperspecial twisted current algebra, where
$\bxi$ is a tuple of partitions indexed by the positive roots. We give three equivalent presentations of these modules. In Section~\ref{section3}, we consider a particular
choice of partitions and relate the corresponding module to twisted Demazure modules. Moreover, we show that the defining relations of these modules
can be greatly simplified. In Section~\ref{section4} we prove the same results for the special twisted current algebras.
In Section~\ref{section5} we give a tensor product decomposition of twisted Demazure modules and in Section~\ref{section6} we introduce the notion of fusion products for twisted modules and prove a fusion product decomposition of twisted Demazure modules. In Section~\ref{section7}, we show that there exists a short exact sequence of graded $\CG$--modules corresponding to the twisted $Q$-–system defined in \cite{HKOTT02}. 

\textit{Acknowledgement: Part of this work was done when both authors were visiting the Centre de recherches mathématiques at the University of Montreal
for the semester program on New directions in Lie theory. We thank the organizers of the semester for this opportunity. 
}

\section{Preliminaries}\label{section1}
\subsection{} We denote the set of complex numbers by $\bc$ and, respectively, the set of integers, non--negative integers, and positive integers  by $\bz$, $\bz_+$, and $\bn$. Moreover, let $\bold N=\{(r,s): r,s\in \frac{1}{2}\bn, s+r\in\bn \}$. Unless otherwise stated, all the vector spaces considered in this paper are $\bc$-vector spaces and $\otimes$ stands for $\otimes_\bc$.


\subsection{}
For a Lie algebra $\lie a$, we let $\bu(\lie a)$ be the universal enveloping algebra of $\lie a$ and denote by $\lie a[t]=\lie a\otimes \bc[t]$ the current algebra associated to $\lie a$.  If, in addition, $\lie a$ is $\mathbb Z_+$-graded, then $\bu(\lie a )$ acquires the unique compatible $\mathbb Z_+$-graded algebra structure. We shall be interested in $\mathbb Z$-graded  representations $V=\oplus_{r\in\mathbb Z} V[r]$ of  $\mathbb Z_+$-graded Lie algebras $\lie a=\oplus_{r\in\bz_+} \lie a[r]$. Clearly, $\lie a[0]$ -- the homogeneous  component of $\lie a$ of grade zero -- is a Lie subalgebra  of $\lie a$ and if $V$ is a $\mathbb Z$-graded representation, then every homogeneous component $V[r]$ is a $\lie a[0]$--module. A morphism between graded $\lie a$-representations is a grade preserving map of $\lie a$-modules.


\subsection{}  We refer to \cite{K90} for the general theory of affine Lie algebras. Throughout, $\widehat{A}$ will denote an indecomposable affine Cartan matrix,  and $\widehat S$  will denote the corresponding Dynkin diagram with the labeling of vertices as in Table Aff2 from \cite[pg.54--55]{K90}. 
Let $S$ be the Dynkin diagram obtained from $\widehat S$ by dropping the zero node and let $A$ be the Cartan matrix, whose Dynkin diagram is $S$.

Let $\widehat{\lie g}$ and $\lie g$  be the  affine Lie algebra and the finite--dimensional  algebra associated to $\widehat A$ and $A$, respectively. We shall realize $\lie g$ as a subalgebra of $\widehat{\lie g}$. We fix $\lie h\subset \widehat{\lie h}$ Cartan subalgebras of $\lie g $ and respectively $\widehat{\lie g}$. We denote by $\widehat{R}$ and, respectively,  $R$ the set of roots of $\widehat{\lie g}$ with respect to $\widehat{\lie h}$, and the set of roots of $\lie g $ with respect to ${\lie h}$. We fix $\widehat\Delta=\{\alpha_0,\dots,\alpha_n\}$  a basis  for $\widehat R$ such that $\Delta=\{\alpha_1,\dots,\alpha_n\}$ is a basis for $R$. The corresponding sets of positive and negative roots are denoted as usual by $\widehat R^\pm$ and respectively $R^\pm$. For $\alpha\in \widehat{R}$, let $\alpha^{\vee}$ be the corresponding coroot. We fix  $d\in \widehat{\lie h}$ such that $\alpha_0(d)=1$ and $\alpha_i(d)=0$ for $i\neq 0$; $d$ is called the scaling element and it is unique modulo the center of $\widehat{\lie g}$. For $1\le i\leq n$, define  $\omega_i\in\H^*$ by $\omega_i(\a_j^\vee)=\d_{i,j}$, for $1\leq j\leq n$, where $\d_{i,j}$ is Kronecker's delta symbol. The element $\omega_i$ is the fundamental weight of $\G$ corresponding to $\a_i^\vee$. We also define $\Lambda_0\in \Haff^*$ by $\Lambda_0(\a_j^\vee)=\d_{0,j}$, for $0\leq j\leq n$, and $\Lambda_0(d)=0$. The element $\Lambda_0$ is the fundamental weight of $\Gaff$ corresponding to $\a_0^{\vee}$. Let $(,)$ be the standard invariant form on $\widehat{\lie h}^{*}$ and for $\alpha\in \widehat{R}^+$ we set 
$$d_{\alpha}=\frac{(\alpha,\alpha)}{2},\ d_{\alpha_i}:=d_i.$$

Let $R_\ell$ and $R_s$ denote respectively the subsets of $R$ consisting of the long and short roots and denote by $R_\ell^\pm, R_s^\pm$ the corresponding subsets of positive and negative roots. 
In this paper we are mainly interested in twisted affine Lie algebras. We set
$$m=\begin{cases}
2,& \text{if $\widehat{\lie g}$ is of type $\tt A^{(2)}_{2n}\ (n\geq 1), \tt A^{(2)}_{2n-1}\ (n\geq 3), \tt D^{(2)}_{n+1}\ (n\geq 4)\mbox{ or }\tt E^{(2)}_{6}$}\\
3,& \text{if $\widehat{\lie g}$ is of type $\tt D^{(3)}_{4}$.} 
\end{cases}
$$
Note that, $d_\alpha=m$ if $\alpha$ is long and $1$ if $\alpha$ is short. We recall the root system of twisted affine algebras.
If $\d$ denotes the unique non-divisible positive imaginary root in $\widehat{R}$, then we have $\widehat{R}=\widehat{R}^+\cup\widehat{R}^-$, where $\widehat{R}^-=-\widehat{R}^+$, $\widehat R^+ =\widehat R^+_{\rm {re}}\cup \widehat R^+_{\rm{im}}$, $\widehat R^+_{\rm{im}} =\mathbb N\delta$, and
$$\widehat{R}^+_{\rm{re}}= R^+\cup(R_s+\mathbb N\delta)\cup(R_\ell+m\mathbb N\delta),\ \mbox {if $\widehat{\lie g}$ is not of type $\tt A^{(2)}_{2n}$},$$
and else
$$\widehat{R}^+_{\rm{re}}=R^+\cup(R_s+\mathbb N\delta)\cup(R_\ell+2\mathbb N\delta)\cup\frac12(R_\ell+(2\mathbb Z_++1)\delta).$$
For $\tt A^{(2)}_{2}$, by convention, $R_s+\mathbb N\delta=\emptyset$.

We also need to consider the set 
$$\widehat R_{\rm re}(\pm)=R^\pm\cup(R_s^\pm+ \bn\delta)\cup(R_\ell^\pm+m\mathbb N\delta),\ \mbox {if $\widehat{\lie g}$ is not of type $\tt A^{(2)}_{2n}$}$$

and else

$$\widehat R_{\rm re}(\pm)=R^\pm\cup(R_s^\pm+ \bn\delta)\cup(R_\ell^\pm+2\mathbb N\delta)\cup\frac12(R_\ell^\pm+(2\mathbb Z_++1)\delta).$$
Remark that $\widehat R_{\rm re}(+)\cup \widehat R_{\rm re}(-)= \widehat{R}^+_{\rm{re}}\cup R^-$.


\subsection{} 

Let $Q=\oplus_{i=1}^n \bz \a_i$ be the root lattice of $R$ and let $\bar{Q}=\oplus_{i=1}^{n-1} \bz \a_i\oplus \frac{1}{2}\bz \a_n$. Let $Q^+$ and $\bar{Q}^+$ be the respective $\bz_+$--cones. 
The weight lattice (resp. coweight lattice) of $R$ is denoted by $P$ (resp. $P^\vee$) and the cone of dominant weights is denoted by $P^+$. Let $\widehat W$ and $W$ be the Weyl groups of  $\widehat{\lie g}$ and $\lie g$ respectively, then 
$\widehat{W}=W\ltimes t_{\bar{Q}}$ if $\widehat{\lie g}$ is of type $\tt A^{(2)}_{2n}$ and else $\widehat{W}=W\ltimes t_{Q}$, where the translation $t_\mu\in \widehat{W}$ for an element $\mu \in Q$ (resp. $\mu \in \bar Q$) is defined by
$$t_{\mu}(b\Lambda_0+\lambda)=b\Lambda_0+\lambda+b\mu \ \ \ (\text{mod} \ \mathbb{C}\delta),\ \lambda\in \lie h^{*}\oplus \bc \delta,\ b\in \mathbb{C}.$$
For a real root $\alpha$ we denote by $w_{\alpha}$ the reflection associated to $\alpha$, then $\widehat W$ respectively $W$ is generated by the reflections associated to the roots $\widehat{\Delta}$ respectively $\Delta$. Furthermore, let $w_0$ be the unique longest element in $W$. The \textit{extended affine Weyl group} $\widetilde{W}$ of $\widehat{\lie g}$ is the semidirect product $\widetilde{W}=W\ltimes t_L$, where $L$ is the coweight lattice if $\widehat{\lie g}$ is of type $\tt A^{(2)}_{2n}$ and else the weight lattice. For an element $\mu \in L$, $t_\mu\in \widetilde{W}$ is defined similarly.
The affine Weyl group is a normal subgroup of $\widetilde{W}$ and the group $\widetilde{W}$ is the semidirect product
$$\widetilde{W}=\Sigma\ltimes \widehat{W},$$
where $\Sigma=\{w\in \widetilde{W}\mid w(\widehat{\Delta})\subseteq \widehat{\Delta}\}$ and each element in $\Sigma$ is an automorphism of the Dynkin diagram. For more details we refer to \cite{Wa01}.
Finally, we remark that $t_{-\mu}\in\widehat{W}$ for all $\mu\in P^+$ if $\widehat{\lie g}$ is of type $\tt A^{(2)}_{2n}$ and else $t_{-\mu}\in\widetilde{W}$ for all $\mu\in P^+$.

\subsection{}
Given $\alpha\in \widehat R^+$ let $\widehat{\lie g}_\alpha\subset\widehat{\lie g} $ be the corresponding root space; note that  $\widehat{\lie g}_\alpha \subset \lie g$ if $ \alpha\in R$. For a non-imaginary root $\alpha$ we denote by $x_{\alpha}$ the generator of $\widehat{\lie g}_\alpha$. We  define several subalgebras of $\widehat{\lie g}$ that will be needed in the rest of the paper. Let $\widehat{ \lie b}$ be the Borel subalgebra corresponding to $\widehat{R}^+$, and let $\widehat{\lie n}^+$ be its nilpotent radical,
 $$\widehat{\lie b}=\widehat{\lie h}\oplus \widehat{\lie n}^+,\ \  \ \ \widehat{\lie n}^\pm =\oplus_{\alpha\in\widehat R^+}\widehat{\lie g}_{\pm \alpha}.$$ The subalgebras $\lie b$ and $\lie n^\pm$ of $\lie g$ are analogously  defined. 

Consider the algebra
$$\lie k=(\lie h\oplus\mathbb Cd)\oplus\widehat{\lie n}^+\oplus\lie n^-.$$ 
 
The twisted current algebra $\CG$ is defined to be the ideal of $\lie k$ defined as
 $$
 \CG=\lie h\oplus\widehat{\lie n}^+\oplus\lie n^-.
 $$
and has a triangular decomposition $$\CG=\lie C\lie n^+\oplus \CH \oplus\lie C\lie n^-,$$ where  
$$\CH=\CH_+\oplus \H, \ \ \ \CH_+=\bigoplus_{k>0}\widehat{\lie g}_{k\delta},\ \ \ \lie C\lie n^\pm=\bigoplus_{\alpha\in \widehat{R}_{\rm re}(\pm)}\widehat{\lie g}_{\pm\alpha}.$$ Note that $\CH$ is an abelian Lie subalgebra.
Following \cite{CIK14} we call $\CG$ the \textit{hyperspecial} twisted current algebra if $\widehat{\lie g}$ is of type $\tt A_{2n}^{(2)}$ and else the \textit{special} twisted current algebra. 
The definition of the hyperspecial twisted current algebra is different from the notion of twisted current algebra of type $\tt A_{2n}^{(2)}$ that exists in the literature (for example, as in \cite{FK11}).
The differences are clarified in \cite[Remark 2.5]{CIK14}.


\subsection{}\label{ev0} 
 
 The element $d$ defines a $\mathbb Z_+$--graded Lie algebra structure on $\CG$: for $\alpha\in \widehat R$ we say that $\widehat{\lie g}_\alpha$ has grade $k$ if $$[d,x_\alpha]=kx_\alpha$$ 
 or, equivalently, if $\alpha(d)=k$. Remark that since $\delta(d)\in\{1,2\}$ the eigenvalues of $d$ are all integers and if $\widehat{\lie g}_\alpha\subset\CG$, then the eigenvalues are non--negative integers.  With respect to this grading, the zero homogeneous component of the twsited current algebra is  $\CG[0]= \lie g$ and the subspace spanned by the positive homogeneous components is an ideal denoted by $  \CG_+$. We have a short exact sequence of Lie algebras,
$$0\to \CG_+\to\CG \stackrel{\ev_0}{\longrightarrow} \lie g\to 0.$$
Note that this exact sequence is right-split but not left-split as a sequence of Lie algebras but it is a split sequence as a sequence of $\lie g$-modules. Clearly the pull--back of any $\lie g$--module $V$ by  $\ev_0$ defines the structure of a graded $\CG$--module on $V$ and we denote this module by $\ev_0^*V$.

\subsection{}  
For $\lambda\in P^+$, denote  by $V(\lambda)$ the irreducible finite--dimensional $\lie g$--module generated by an element $v_\lambda$ with defining relations $$\lie n^+ v_\lambda=0,\ \ hv_\lambda=\lambda(h) v_\lambda,\ \ (x_{-\alpha})^{\lambda(\alpha^{\vee})+1} v_\lambda=0,\ \ h\in\lie h,\ \alpha\in R^+. $$ 
It is well--known that any irreducible $\lie g$--module is isomorphic to $V(\lambda)$ for some $\lambda\in P^+$ and $V(\lambda)\cong V(\mu)$ iff $\lambda=\mu$. Moreover,  any finite--dimensional $\lie g$--module is isomorphic to a direct sum of  modules $V(\lambda)$, $\lambda\in P^+$.  
If $V$ is a $\lie h$--semisimple $\lie g$--module (in particular if $\dim V<\infty$),  we have $$ V=\bigoplus_{\mu\in\lie h^*}V_\mu,\ \ V_\mu=\{v\in V: hv=\mu(h)v,\ \ h\in\lie h\},$$ and we set $\wt V=\{\mu\in\lie h^*: V_\mu\ne 0\}.$  If $\dim V_\mu<\infty$ for all $\mu\in \wt V$, then we define $\ch_{\lie h} V:\lie h^*\to \bz_+$, by sending $\mu\to\dim V_\mu$. If $\wt V$ is a finite set, then $$\ch_{\lie h}V=\sum_{\mu\in\lie h^*}\dim V_\mu e(\mu)\in\bz[P].$$

\subsection{}
The methods we use for hyperspecial twisted current algebras differ from the methods for special twisted current algebras and therefore we shall regard the hyperspecial case separately.\par \textit{Unless otherwise stated, we consider from now on the twisted affine algebra of type $\tt A^{(2)}_{2n}$.} So $A$ is the Cartan matrix of type $\tt A_1$ if $n=1$ and of type $\tt C_n$ if $n\geq 2$. For $n=1$, by convention, $R_\ell=R$ and $R_s=\emptyset$.

We recall an explicit construction of the algebra
$\CG$ as a subalgebra of $\mathcal{L}(\mathfrak{sl}_{2n+1})=\big(\mathfrak{sl}_{2n+1}\otimes \bc[t^{\pm}]\big)$ from \cite[Section 4]{CIK14}. 
We fix a Chevalley basis $\{X^\pm_{i,j}$, $H_i$ ~:~$1\le i\le j\le 2n\}$  for $\lie{sl}_{2n+1}$.
Let $\alpha\in R^+_s$ and $r\in \bz_+$;  $\alpha$ is necessarily of one of the two forms listed below for some $1\le i\le j< n$. We set
\begin{equation*}
\begin{aligned}
&x_{\pm\alpha+r\delta}= X^\pm_{i,j}\otimes t^r+(-1)^{i+j}X^\pm_{2n+1-j,2n+1-i}\otimes (-t)^r, &&\text{for } \alpha=\sum_{s=i}^j\alpha_s,\\
&x_{\pm\alpha+r\delta}=X^{\pm}_{i,2n-j}\otimes t^{r\pm 1}+(-1)^{i+j}X^{\pm}_{j+1,2n+1-i}\otimes (-t)^{r\pm 1}, &&\text{for }  \alpha=\sum_{s=i}^{j}\alpha_s + 2\sum_{s=j+1}^{n-1}\alpha_s+\alpha_n.
\end{aligned}
\end{equation*}
Let  $\alpha\in R^+_\ell$ and $r\in \bz_+$;  $\alpha$ is of the form $2(\alpha_{i}+\cdots +\alpha_n)-\alpha_n$ for some $1\le i\le n$. We set
\begin{equation*}
\begin{aligned}
& x_{\pm\alpha+2r\delta}= X_{i,2n+1-i}^\pm \otimes t^{2r\pm 1}, \\
&x_{\frac12(\pm\alpha+(2r+1)\delta)}= X^\pm_{i,n}\otimes t^{(2r+1\pm 1)/{2}}+(-1)^{i} X^\pm_{n+1, 2n+1-i}\otimes (-t)^{(2r+1\pm 1)/2}. 
\end{aligned}
\end{equation*}
Finally, for $1\le i\le n$,  we set $$h_{i,r\delta}= H_i\otimes t^r+H_{2n+1-i}\otimes (-t)^r.$$ We remark that $\alpha_i^\vee=h_{i,0}$ for $1\le i\le n$.
\subsection{}
The following proposition is needed later in this paper.

\begin{prop}\label{sp2}
We have
\begin{enumerit} 
\item[(i)] For $\alpha\in R^+$, the subalgebra of $\CG$ generated by the elements $$\{x_{\pm\alpha+r\delta}~:~r\in\mathbb{Z}_+, \a+r\delta\in \widehat{R} \}$$   is isomorphic to $\lie{sl}_2[t]$.
\item[(ii)] For $\alpha\in R^+_\ell$,  the subalgebra generated by the elements $$\{x_{\frac 12(\pm\alpha+(2r+1)\delta)}, x_{\pm\alpha+2r\delta}~:~ r\in\mathbb{Z}_+\}$$ is isomorphic to the current algebra of type $\tt A_2^{(2)}$.
\end{enumerit}
\end{prop}


 
\section{The modules \texorpdfstring{$V(\bxi)$}{V}} \label{section2}
\subsection{}
The aim of this section is to define $\CG$ modules $V(\bxi)$, depending on a tuple of partitions $\bxi=(\xi^{\alpha})_{\alpha\in R^+}$, one partition attached to each positive root $\alpha\in R^+$. Theses modules will be the twisted analogues of the modules studied in \cite{CV13} and are quotients of local Weyl modules for hyperspecial twisted current algebras studied first in \cite{CIK14}. We recall that the local Weyl module $W_{\loc}(\lambda)$, $\lambda$ a dominant integral weight, is the $\CG$-module generated by an element $w_{\l}$ with the following relations

\begin{align}\label{lWeylrels}
&(\CN^+\oplus \CH_+) \cdot w_{\l}=0, \hspace{-0,8cm}
&h\cdot w_{\l}=\l(h)w_{\l},~h\in\H, \ \
&x_{-\a}^{\l(\a^\vee)+1}\cdot w_{\l}=0, ~ \a\in R^+ .
\end{align}

Note that since the ideal is homogeneous the module $W_{\loc}(\lambda)$ is graded. 
\subsection{}
For a dominant integral weight $\lambda\in P^+$, we say that
 $\bxi=(\xi^{\alpha})_{\alpha\in R^+}$ is a $\lambda$--compatible $|R^+|$--tuple of partitions, if  $$\xi^{\alpha}=(\xi^{\alpha}_{0}\ge\xi^{\alpha}_{1}\ge\cdots\ge\xi^{\alpha}_{s}\ge\cdots\ge 0),\ \ |\xi^\alpha|=\sum_{j\ge 1}\xi^{\alpha}_j=\lambda(\alpha^{\vee}). $$
If $|\xi^\alpha|>0$, then let $s_{\alpha}$ be the number of non-zero parts of $\widetilde{\xi}^{\alpha}:=(\xi^{\alpha}_{1}\ge\cdots\ge\xi^{\alpha}_{s}\ge\cdots\ge 0)$. For $\alpha\in R^+_{\ell}$, $k\in \bz_+$ we set 
$$\phi(\xi^{\alpha};k)=\begin{cases}
-\frac{1}{2}\xi^{\alpha}_{k+1}+\sum_{j\geq k+1} \xi^{\alpha}_{j},& \text{ if $0\leq k\leq s_{\alpha}-2$ }\\
(\xi^\a_{s_{\alpha}}-\frac{1}{2}\xi^\a_{s_{\alpha}-1})_{+},& \text{ if $k=s_{\alpha}-1$}\\
\ \  0 & \text{ else,}\end{cases}
$$
where we understand $s_+=\max\{0,s\}$ for $s\in \mathbb{R}$. For any non--negative integer $b$ and $x\in \CG$ let $x^{(b)}:=\frac{1}{b!}x^b$.
Using the above data we define
 $V(\bxi)$ to be the graded quotient of $W_{\loc}(\lambda)$ by the submodule generated by the graded elements:
  \begin{equation*}
 \Big\{(x_{\alpha+d_\alpha \delta})^{(s)}(x_{-\alpha})^{(s+r)}w_\lambda : \alpha\in R^+, \ s,r\in\bn,  \   s+r\ge 1+ rk+\sum_{j\ge k+1}\xi^{\alpha}_j,\ \ {\rm{for\ some}}\ k\in\mathbb{Z}_+ \Big\}
\end{equation*}
 \begin{equation*}
\bigcup \Big\{ (x_{\frac{\alpha}{2}+\frac{\delta}{2}})^{(2s)}(x_{-\alpha})^{(s+r)}w_\lambda : \alpha\in R^+_{\ell}, \ (r,s)\in\bold N, \    s\ge \frac{1}{2}+ 2rk+\phi(\xi^{\alpha};k),\ \ {\rm{for\ some}}\ k\in\mathbb{Z}_+\Big\}.\end{equation*}
\vskip 6pt
\noindent
Denoting by $v_\bxi$  the image of $w_\lambda$ in $V(\bxi)$, it is clear that
$V(\bxi)$ is  the graded $\CG$--module generated by $v_\bxi$ with defining relations:
\begin{gather}\label{vxi}(\CN^+\oplus \CH_+) v_\bxi=0,\ \ h\cdot v_\bxi=\l(h)v_\bxi,~h\in\H,\
 x_{-\a}^{\l(\a^\vee)+1}\cdot v_\bxi=0, ~ \a\in R^+,
 \end{gather} 
and for $r,s\in\bn$ and $k\in \mathbb{Z}_+$ we have
\begin{align}\label{firstshort}
&(x_{\alpha+d_{\alpha}\delta})^{(s)}(x_{-\alpha})^{(s+r)}v_\bxi=0, \ \alpha\in R^+, \ s+r\ge 1+ rk+\sum_{j\ge k+1}\xi^{\alpha}_j,  
\end{align}
and for $(r,s)\in\bold N$ and $k\in \mathbb{Z}_+$ we have
\begin{align}\label{thirdlong}
&(x_{\frac{\alpha}{2}+\frac{\delta}{2}})^{(2s)}(x_{-\alpha})^{(s+r)}v_\bxi=0,  \ \alpha\in R^+_{\ell}, \ s+r\ge \frac{1}{2}+(2k+1)r+\phi(\xi^{\alpha};k).
\end{align}
We will later see that $V(\bxi)$ is a non-zero indecomposable module for any $\bxi$.
\vskip 6pt

\subsection{ }
By using Garland type identities for the hyperspecial twisted current algebra $\CG$ and the current algebra $\mathfrak{sl}_2[t]$ we will give three equivalent presentations of the modules $V(\bxi)$. The idea here is to reformulate the relations \eqref{firstshort} and \eqref{thirdlong}.
First we fix some notations.

For $s,r\in\bz_+$, let \begin{equation*}\label{maxs}\bs(r,s)=\Big\{\bold b=(b_p)_{p\ge 0}: b_p\in\bz_+, \ \  \sum_{p\ge 0} b_p=r,\ \ \sum_{p\ge 0} pb_p=s\Big\}. \end{equation*}
Given  $\alpha \in R^+$ and   $s,r\in\bz_+$,
define  elements $\bx^-_\alpha(r,s)\in\bu(\CG)$ by, \begin{equation}\label{xlm}\bx^-_\alpha(r,s)=\sum_{\bold b\in\bs(r,s)}(x_{-\alpha})^{(b_0)}(x_{-\alpha+d_\alpha \delta})^{(b_1)}\cdots (x_{-\alpha+ d_\alpha s\delta})^{(b_s)},\end{equation}where we understand $\bx^-_\alpha(r,s)=0$ if $\bs(r,s)=\emptyset$.
In particular,  \begin{equation*}\label{special1} \bx^-_\alpha(0,s)=\delta_{s,0},\ \bx^-_\alpha(1,s)=x_{-\alpha+d_\alpha s\delta}.\end{equation*}

The following result is a combination of Proposition~\ref{sp2}(i) and a result of Garland \cite{G78}.

\begin{lem}\label{garsl_2} {{Given  $s\in\bn $}}, $r\in\bz_+$ and $\alpha\in R^+$ we have, $$(x_{\alpha+d_\alpha \delta})^{(s)}(x_{-\alpha})^{(s+r)}-(-1)^{s}\bx^-_\alpha(r,s)\in\bu(\CG)\CN^+.$$\hfill\qedsymbol
\end{lem}

\subsection{}
For any $(r,s)\in\bold N$ we shall define a second indexing set (see also \cite[Section 7]{CIK14}), namely let $\tilde{\bs}(r,s)$ be the set of non-negative integer sequences $\bold p=(p_i)_{i\in \frac{1}{2}\bz_+}$ that satisfy
\begin{align}\label{a22sequences}
&r = \sum_{N\geq 0} \frac{1}{2}p_{N+\frac{1}{2}} +\sum_{N\geq 0} p_N,\\
&s = \sum_{N\geq 0} \frac{2N+1}{2}p_{N+\frac{1}{2}} +\sum_{N\geq 0} 2Np_N.
\end{align}
The support of $\bold p\in \tilde{\bs}(r,s)$ is defined as 
$$
\supp(\bold p)=\big\{i\in \frac{1}{2}\bz_+~|~p_i\neq 0\big\}.
$$
Also, let 
\begin{equation}\label{ydefn}
\bold y^\a(r,s)=\sum_{\bold p\in \tilde{\bs}(r,s)} \overset{\rightarrow}{\prod}_{N\geq 0}
\left(
\frac{(-1)^{p_{N+\frac{1}{2}}}}{2^{Np_{N+\frac{1}{2}}}}
{x_{-\frac{\a}{2}+(N+\frac{1}{2})\d}^{(p_{N+\frac{1}{2}})}}
\right)\left(
\frac{(-1)^{p_{N}}(2-(-1)^N)^{p_N}}{2^{2Np_{N}}}
{x_{-\a+2N\d}^{(p_N)}}
\right),
\end{equation}
where $\overset{\rightarrow}{\prod}_{N\geq 0}$ refers to the product of the specified factors written exactly in the increasing order of the indexing parameter (the factors do not commute and the order in which they appear in the product is important). The following lemma is proven in \cite[Corollary 5.39.]{F-V98}.

\begin{lem}\label{gara_2}
Let $(r,s)\in\bold N$. Then,
$$
(-1)^{s+r}{x_{\frac{\a}{2}+\frac{1}{2}\d}^{(2s)}}{x_{-\a}^{(s+r)}}-\bold y^\a(r,s)\in \bu(\CG)\CN^+ .
$$\hfill\qedsymbol
\end{lem}

\subsection{}
In order to obtain a second presentation we shall reformulate \eqref{firstshort}--\eqref{thirdlong}. Using Lemma~\ref{garsl_2} we see that,
$\big( (x_{\a +d_\a \d})^{(s)} (x_{-\a})^{(s+r)}-(-1)^s\bx^-_\alpha(r,s)\big)v_\bxi =0,$ and hence \eqref{firstshort} is equivalent to:
\begin{equation}\label{first1}\bx^-_\alpha(r,s)v_\bxi =0,\ \alpha\in R^+,\ \forall \ s,r\in \bn,\ k\in\mathbb{Z}_+ \ \text{such that} \ s+r\ge 1+ rk+\sum_{j\ge k+1}\xi^{\alpha}_j.\end{equation}
Now using Lemma~\ref{gara_2} we see that,
$\big((-1)^{s+r}{x_{\frac{\a}{2}+\frac{1}{2}\d}^{(2s)}}{x_{-\a}^{(s+r)}}-\bold y^\a(r,s)\big)v_\bxi =0,$
and hence \eqref{thirdlong} is equivalent to
\begin{equation}\label{second2}\bold y^\a(r,s)v_\bxi =0,\ \alpha\in R_{\ell}^+, \  \forall \ (s,r)\in \bold N, \ k\in\mathbb{Z}_+ \ \text{such that} \ s+r\ge \frac{1}{2}+(2k+1)r+\phi(\xi^{\alpha};k). \end{equation}
It follows that $V(\bxi)$ is the quotient of $W_{\loc}(\lambda)$ by the submodule generates by \eqref{first1} and \eqref{second2}. From now on, we use both presentations of $V(\bxi).$ As a first application we show that $V(\bxi)$ is non-zero and hence indecomposable.
\begin{prop}
Let $\bxi$ be a $\lambda$-compatible partition.
The module $V(\bxi)$ has a unique irreducible quotient $\ev^*_0V(\lambda)$, and hence $V(\bxi)$ is a non-zero indecomposable $\CG$--module. 
\begin{pf}
We denote by $v_{\lambda}$ the highest weight vector of $\ev^*_0V(\lambda)$. Obviously $\ev^*_0V(\lambda)$ is a quotient of the local Weyl module $W_{\loc}(\lambda)$. If $\mathbf b\in \mathbf S(r,s)$ and $s>0$, then $b_p\neq 0$ for some $p\neq 0$ and hence $\bold x^{-}_{\alpha}(r,s)v_{\lambda}=0.$
Now let $(r,s)\in \bold N$, $k\in \mathbb{Z}_+$ such that $r+s\geq \frac{1}{2}+(2k+1)r+\phi(\xi^{\alpha};k)$. It follows $s\neq 0$ and thus there exists for any $\mathbf p\in \tilde{\mathbf S}(r,s)$ a positive integer $N$, such that either $p_N\neq 0$ or $p_{N-\frac{1}{2}}\neq 0$. It follows by the definition of $\bold y^\a(r,s)$ in \eqref{ydefn} that $\bold y^{\a}(r,s)v_{\lambda}=0$. Consequently, $V(\bxi)$ has a unique non-zero quotient and is indecomposable. 
\end{pf}
\end{prop}

\subsection{} The third presentation requires an alternative description of the sets $\bs(r,s)$ and $\tilde{\bs}(r,s)$. First we focus on $\bs(r,s)$ and follow the methods from \cite{CV13}.
For $k\in\bz_+$, let $\bs(r,s)_k$ (resp. $_k\bs(r,s)$) be the subset of $\bs(r,s)$ consisting of elements $(b_p)_{p\ge 0}$, satisfying $$b_p=0,\ \ p\ge k,\ \ ({\rm{resp.}}\ \ b_p=0\ \ p<k).$$ 
Clearly,
\begin{equation}\label{alts}\bs(r,s)={} _k\bs(r,s)\bigsqcup \left(\bs(r-r',s-s')_k\times {}_{k}\bs(r',s')\right), \end{equation} where the union is over all pairs $r',s'\in\bz_+$ with $\bs(r-r',s-s')_k\ne\emptyset$ and ${}_k\bs(r',s')\ne\emptyset$.

\vskip 6pt

For $\a \in R^+$, define elements  $\bx^-_\a(r,s)_k$ and $_k\bx^-_\a(r,s)$ of $\bu(\CG)$ as in \eqref{xlm}, with the important difference that the index set for the summation is $\bold S_k(r,s)$ and ${}_{k}\bold S(r,s)$ respectively.
It follows that\begin{equation}\label{specialcase} \bx^-_\a(r,s)_k\ne 0\implies s\le (k-1)r,\quad _{k}\bx^-_\a(r,kr)=(x_{-\a + d_\a k\d})^{(r)} .\end{equation}

With the above notation we can state the following proposition, which is proven in \cite{CV13} and allows us to reformulate \eqref{first1}.

\begin{prop}\label{max}  Let $V$ be any representation of $\CG$ and let $v\in V$, $\a \in R^+$ and $K\in\bz_+$. Then,
\begin{gather*}\bx^-_{\alpha}(r,s)v=0\ {\rm{for\ all}}\ s,r\in\bn,\ k\in \mathbb{Z}_+
 \ \ {\rm{with}}\  s+r\ge 1+kr+K\ \iff\\  _{k}\bx^-_{\alpha}(r,s)v=0  \ {\rm{ for \ all}}\  s,r \in\bn,\ k\in \mathbb{Z}_+ \  {\rm{with}}\ s+r\ge 1+ kr+K.\end{gather*}\hfill\qedsymbol
\end{prop}
Our aim is to obtain a similar reformulation for \eqref{second2}. In order to do that we shall give an alternative description of the indexing set $\tilde{\bs}(r,s)$ and prove an analogous result to Proposition~\ref{max} for $\bold y^\a(r,s).$ 
\subsection{}
For $(r,s)\in\bold N$ and $k\in \frac{1}{2}\bz_+$, let 
$$
\tilde{\bs}_{<k}(r,s)=\{\bold p\in \tilde{\bs}(r,s)~|~\supp(\bold p)\subseteq [0,k) \}$$ and $$\tilde{\bs}_{\geq k}(r,s)=\{\bold p\in \tilde{\bs}(r,s)~|~\supp(\bold p)\subseteq [k,\infty) \}.
$$
We define $\bold y^\a_{<k}(r,s)$ and $\bold y^\a_{\geq k}(r,s)$ as in \eqref{ydefn},
with the difference that the index set for the summation is $\tilde{\bs}_{<k}(r,s)$ and $\tilde{\bs}_{\geq k}(r,s)$ respectively.
We record the following simple lemma which will be needed later.
\begin{lem}\label{lemmas} Let $\bold p=(p_i)_{i\in\frac{1}{2}\bz_+}\in \tilde\bs(r,s)$ for some $k\in \bn$.
\begin{itemize}
 \item[(i)] Suppose that $\bold p\in \tilde\bs_{<k+\frac{1}{2}}(r,s)$, then $s\leq (2k-1)r+\sum_{N\leq k}p_N$.
\item[(ii)]Suppose that $\bold p\in \tilde\bs_{<k}(r,s)$, then $s\leq 2(k-1)r+\sum_{N\leq k-1} \frac{1}{2} p_{N+\frac{1}{2}}$. 
\end{itemize}
\end{lem}
\begin{pf}
 Suppose that $\bold p\in \tilde\bs_{<k+\frac{1}{2}}(r,s)$.
 Then
 \begin{align*}
 s &= \sum_{k-1\ge N\geq 0} \frac{2N+1}{2}p_{N+\frac{1}{2}} +\sum_{k\ge N\geq 0} 2Np_N &\\&
 \le (2k-1) \sum_{k-1\ge N\geq 0}\frac{1}{2} p_{N+\frac{1}{2}} +\sum_{k\ge N\geq 0} (2k-1)p_N + \sum_{k\ge N\ge 0}p_N &\\&
 \le (2k-1)r+\sum_{k\ge N\ge 0}p_N.
 \end{align*}
 
Suppose that $\bold p\in \tilde\bs_{<k}(r,s)$. Then
\begin{align*}
 s &= \sum_{k-1\ge N\geq 0} \frac{2N+1}{2}p_{N+\frac{1}{2}} +\sum_{k-1\ge N\geq 0} 2Np_N &\\&
 \le 2(k-1) \sum_{k-1\ge N\geq 0} \frac{1}{2} p_{N+\frac{1}{2}} +\sum_{k-1\ge N\geq 0} 2(k-1)p_N+\sum_{k-1\ge N\geq 0}\frac{1}{2} p_{N+\frac{1}{2}}&\\&
 \le 2(k-1)r+\sum_{k-1\ge N\geq 0} \frac{1}{2}p_{N+\frac{1}{2}}.
 \end{align*}
\end{pf}
\subsection{}
The next lemma gives a clever way of writing the elements $\bold y^\a(r,s)$ and is needed for the proof of the next proposition.
\begin{lem}\label{rearrangea_2} Let $\a \in R^+_\ell.$ Suppose that $(r,s)\in\bold N$, $k\in \mathbb{Z}_+$ are such that $s+r\ge (2k+1)r+K$ for some $K\in\frac{1}{2}\bz_+$.
  Then we have
  \begin{equation}\label{eq4}
\bold y^\a(r,s)=\bold y^\a_{\geq k+\frac{1}{2}}(r,s)+\sum_{(r',s')\in T(r,s,k)}  \bold y^\a_{<k+\frac{1}{2}}(r-r',s-s') \bold y^\a_{\geq k+\frac{1}{2}}(r',s') ,
\end{equation}
 where $T(r,s,k)=\big\{(r',s'):(r',s')\in \bold N, \ r'<r, \ s'\leq s \  \text{and}\ s'+r'\ge (2k+1)r'+K\big\}.$

  \end{lem}
  \begin{pf}
   If  $(s-s')> 2k(r-r')$ then $\tilde{\bs}_{< k+\frac{1}{2}}(r-r',s-s')=\emptyset$ by Lemma~\ref{lemmas}, in which case $\bold y^\a_{< k+\frac{1}{2}}(r-r',s-s')=0$. Hence $s-s'\leq 2k(r-r')$ and
 $$s'+r'+(s-s')+(r-r')\ge (2k+1)(r-r')+(2k+1)r'+K$$ implies that $s'+r'\ge (2k+1)r'+K.$ The remaining restrictions on $(r',s')$ are obvious (see \cite[Section 8.2]{CIK14}).
  \end{pf}


Now we are able to prove the analog result to Proposition~\ref{max}.
\begin{prop}\label{maxa_2}  Let $V$ be any representation of $\CG$ and let $v\in V$, $\a \in R^+_\ell$ and $K\in\frac{1}{2}\bz_+$. Then,
\begin{gather*}\bold y^{\alpha}(r,s)v=0\ {\rm{for\ all}}\ (r,s)\in\bold N, k\in \mathbb{Z}_+
 \ \ {\rm{with}}\  s+r\ge \frac{1}{2}+(2k+1)r+K\ \iff\\  \bold y^{\alpha}_{\ge k+\frac{1}{2}}(r,s)v=0  \ {\rm{ for \ all}}\ (r,s)\in\bold N, k\in \mathbb{Z}_+ \  {\rm{with}}\ s+r\ge \frac{1}{2}+(2k+1)r+K.\end{gather*}
\end{prop}

\begin{pf} First suppose that $\bold y^{\alpha}(r,s)v=0$ for all $(r,s)\in\bold N, k\in \mathbb{Z}_+ \ {\rm{with}}\  s+r\ge \frac{1}{2}+(2k+1)r+K$.   
We shall prove the statement by induction on $r$.
For the initial step assume $r=\frac{1}{2}$ and let $s\in\frac{1}{2}\bn$ be such that $2s+1\ge 1+(2k+1)+2K$. Then we have $s\ge k+\frac{1}{2}$ and hence
$$\bold y^{\alpha}(1/2,s)=\bold y^{\alpha}_{\ge k+\frac{1}{2}}(1/2,s),$$
which proves $\bold y^{\alpha}_{\ge k+\frac{1}{2}}(1/2,s)v=0$.
Assume now that we have proved the statement for all $r'<r$ and all $s'\le s$ with $s'+r'\ge \frac{1}{2}+(2k+1)r'+K.$
Using Lemma~\ref{rearrangea_2}, we get 
$$\bold y^\a(r,s)v=\bold y^\a_{\geq k+\frac{1}{2}}(r,s)v+\sum_{(r',s')\in T(r,s,k)}  \bold y^\a_{<k+\frac{1}{2}}(r-r',s-s') \bold y^\a_{\geq k+\frac{1}{2}}(r',s')v ,$$
where the sum is over all $r'<r$ and $s'\le s$ with $s'+r'\ge \frac{1}{2}+(2k+1)r'+K$. The inductive hypothesis applies to the second term on the RHS and hence we get  $\bold y^\a_{\geq k+\frac{1}{2}}(r,s)v=0$. The converse statement is obvious by Lemma~\ref{rearrangea_2}.
\end{pf}


\subsection{} We can now give the third presentation of $V(\bxi)$.
\begin{prop}\label{third} The module $V(\bxi)$ is generated by the element $v_\bxi$ with defining relations \eqref{vxi} and
\begin{align}\label{third1}&_{k}\bx^-_{\alpha}(r,s)v_\bxi=0,\  {\alpha}\in R^+,\  s,r\in\bn,\ k\in \mathbb{Z}_+ \ \ s+r\ge 1+kr+\sum_{j\ge k+1}\xi^{\alpha}_j, \ \text{and} \\
\label{third2} &\bold y^{\alpha}_{\ge k+\frac{1}{2}}(r,s)v_\bxi=0, \ {\alpha}\in R^+_\ell, \ (r,s)\in\bold N, \ k\in\bz_+, \ \ s+r\ge \frac{1}{2}+(2k+1)r+\phi(\xi^{\alpha};k).\end{align}
\end{prop}
\begin{pf} This is immediate from  Proposition~\ref{max} and Proposition~\ref{maxa_2}. 
\end{pf}
As a corollary we obtain that several elements are contained in the annihilator of $v_\bxi$.

\begin{cor} \label{cthird}
 For all ${\alpha}\in R^+$, $r\in\bn, k\in \bz_+$ with  $r\ge 1+ \sum_{j\ge k+1}\xi^{\alpha}_j$, we have \begin{equation}\label{demreltype}(x_{-\alpha+d_\a k\d})^{r}v_\bxi=0.\end{equation}
 Similarly for all $\a \in R^+_\ell$, $r\in\bn, k\in \bz_+$  with  $r\ge 1+ 2\phi(\xi^\a;k)$, we have \begin{equation}\label{demreltype2}(x_{-\frac{\alpha}{2}+(k+\frac{1}{2})\d})^{r}v_\bxi=0.\end{equation}

\end{cor}

\begin{pf}
In order to prove the first statement, we take  $s=kr$ in \eqref{third1}. Note that there is nothing to prove if $k=0$. We have $s+r\ge 1+kr+ \sum_{j\ge k+1}\xi^{\alpha}_j$ and  using the  second equation in  \eqref{specialcase}, 
 we get $$ _{k}\bx^-_\alpha(r, kr)v_\bxi= (x_{-\alpha+d_\a k\d})^{(r)}v_\bxi=0.$$
To see the second equation, take $s=(2k+1)\frac{r}{2}$. Then  we have $s+\frac{r}{2}\ge \frac{1}{2}+(2k+1)\frac{r}{2}+\phi(\xi^{\alpha};k).$ Now using 
$$\bold y^{\alpha}_{\ge k+\frac{1}{2}}(r/2, (2k+1)r/2)v_\bxi=\frac{(-1)^r}{2^{kr}}(x_{-\frac{1}{2}\alpha+(k+\frac{1}{2})\d})^{(r)}v_\bxi,$$
we get the desired result.\par
\end{pf}

Summarized we have defined indecomposable finite--dimensional graded modules $V(\bxi)$ for tuples of partitions and gave three equivalent presentations of these. In the next section we establish connections to Demazure modules.



\section{Simplified presentation and connection to Demazure modules}\label{section3}

In this section, we shall simplify the defining relations described in the previous section for  special kinds of $\lambda$--compatible partitions. By using the simplified presentation we will make connections to well--known $\CG$--modules, the twisted Demazure modules. These results were known before for the untwisted current algebras in \cite{CV13} and are new for the twisted current algebras. For the special twisted current algebras see Section~\ref{section4}.

\subsection{}  If  $i_1>\cdots> i_r$ are the  distinct  non--zero parts of a partition $\xi$ and  $i_k$ occurs $s_k$ times
 then  we denote this partition for simplicity by $(i_1^{s_1},\dots, i_r^{s_r} )$.
A  partition $\xi$ is said to be rectangular if it is either the empty partition or of the form $(k^m)$ for some $k,m\in\bn$. A partition is said to be a fat hook if it is of the form $(k_1^{s_1}, k_2^{s_2})$ with $k_j,s_j\in\bn$, $j=1,2$ and special fat hook  if $s_2=1$.  We simply call a partition $\xi$
special when it is a special fat hook. For rectangular and special partitions the following theorem gives the simplified relations.

\begin{thm}\label{genmax2} Let $\bxi=(\xi^{\alpha})_{\alpha\in R^+}$ be a $\lambda$--compatible $|R^+|$--tuple of partitions. Assume  that $\xi^{\alpha}$ is  either rectangular or special for $\alpha\in R^+$. Then, $V(\bxi)$ is isomorphic to the quotient of $W_{\loc}(\lambda)$ by the submodule generated by the elements
\begin{align}\label{finer1}&\big\{ (x_{-\alpha+d_{\alpha}s_{\alpha}\delta})w_\lambda:\alpha\in R^+\big\}
\ \bigcup \big\{(x_{-\alpha+d_{\alpha}(s_{\alpha}-1)\delta})^{\xi^\alpha_{s_\alpha}+1}w_\lambda:\alpha\in R^+, \ \ \xi^{\alpha} \ {\rm{{special}}}\big\}\\
\label{finer2}
&\big\{ (x_{-\frac{\alpha}{2}+(s_{\alpha}+\frac{1}{2})\delta})w_\lambda
:\alpha\in R_{\ell}^+\big\}
\bigcup
\big\{(x_{-\frac{\alpha}{2}+(s_{\alpha}-\frac{1}{2})\delta})^{2\phi(\xi^{\alpha};s_{\alpha}-1)+1}w_\lambda:\alpha\in R_{\ell}^+, \ \ \xi^{\alpha} \ {\rm{{special}}}\big\}
\end{align} \vskip 6pt
\end{thm}

\begin{pf}
We denote by $U$ the submodule of $W_{\loc}(\lambda)$ generated by the elements \eqref{finer1} and \eqref{finer2} and
let $\tilde V(\bxi)$ be the corresponding quotient of $W_{\loc}(\lambda)$. Corollary~\ref{cthird} proves that we have a surjective map 
$$\tilde V(\bxi)\longrightarrow V(\bxi)\longrightarrow 0.$$  
In order to prove that we have a surjective map from $V(\bxi)$ to $\tilde V(\bxi)$ we shall prove that all defining relations of $V(\bxi)$ are satisfied in $\tilde V(\bxi)$. We start with the relations in \eqref{firstshort}. For $\alpha\in R^+$ we consider the subalgebra $\mathfrak{sl}_2(\alpha)$ of $\CG$ generated by $\{x_{\pm\alpha+p\delta}~:~p\in\mathbb{Z}_+, \a+p\delta\in \widehat{R} \}$ and let $\bu(\mathfrak{sl}_2(\alpha))$ be the corresponding universal enveloping algebra, which is by Proposition~\ref{sp2}(i) isomorphic to $\bu(\mathfrak{sl}_2[t])$. Furthermore, let $V_{\mathfrak{sl}_2}(\bxi_{\alpha})$ be the $\mathfrak{sl}_2[t]$ module associated to a single partition $(\widetilde{\xi}^{\alpha})$ defined in \cite[Section 2]{CV13}. Since the defining relations of $V_{\mathfrak{sl}_2}(\bxi_{\alpha})$ are simplified in \cite[Theorem 3.1]{CV13} we obtain a surjective map
$$V_{\mathfrak{sl}_2}(\bxi_{\alpha})\twoheadrightarrow \bu(\mathfrak{sl}_2(\alpha))\tilde{v}_{\bxi}.$$
In particular, all relations in $V_{\mathfrak{sl}_2}(\bxi_{\alpha})$ hold in $\bu(\mathfrak{sl}_2(\alpha))\tilde{v}_{\bxi}\subseteq \tilde V(\bxi)$, which implies that \eqref{firstshort} holds in $\tilde V(\bxi)$.
Now we prove that the relations in \eqref{thirdlong} are satisfied, by considering several cases. We start with $2r\ge \xi^{\alpha}_0$ and $k$ is arbitrary. Then we get
$$s+r\geq \frac{1}{2}+(2k+1)r+\phi(\xi^{\alpha};k)\geq \frac{1}{2}+k\xi^{\alpha}_0+\frac{1}{2}\xi^{\alpha}_0+\phi(\xi^{\alpha};k)\geq
\frac{1}{2}+\lambda(\alpha^{\vee}).$$
Thus $\bold y^{\alpha}(r,s)w_{\lambda}=0$, since $s+r\in \bn$. So we can assume from now on $2r < \xi^{\alpha}_0$. By \eqref{finer1} and \eqref{finer2} it is easy to see that 
$$x_{-\alpha+d_\alpha s\delta}w_\lambda, \ \ x_{-\frac{\alpha}{2}+(s+\frac{1}{2})\delta}w_\lambda \in U, \ \ \forall \ s\ge s_\alpha,$$ 
and hence we have
$$\big(\bold y^{\alpha}(r,s)-\bold y^{\alpha}_{< s_{\alpha}}(r,s)\big)w_\lambda \in U.$$
We shall prove that we already have $\bold y^{\alpha}_{< s_{\alpha}}(r,s)w_\lambda\in U$, which would imply $\bold y^{\alpha}(r,s)w_\lambda\in U$. If $k\geq s_{\alpha}$ we get
$s\geq \frac{1}{2}+2kr> 2s_{\alpha}r$ and hence $\bold y^{\alpha}_{<s_{\alpha}}(r,s)=0$ by Lemma~\ref{lemmas}. So in addition to $2r < \xi^{\alpha}_0$ we can assume from now on $k\leq s_{\alpha}-1$.
If $\xi^{\alpha}$ is rectangular, then we obtain
$$s\geq \frac{1}{2}+2kr+\phi(\xi^{\alpha};k)\geq \frac{1}{2}+2kr+(s_{\alpha}-k-1)2r+r\geq \frac{1}{2}+2(s_{\alpha}-1)r+\sum_{0\le N} \frac{1}{2}p_{N+\frac{1}{2}},$$
which gives once more that $\bold y^{\alpha}_{<s_{\alpha}}(r,s)=0$. Summarizing, the claim follows if $2r\geq \xi^{\alpha}_0$ or if $\xi^{\alpha}$ is rectangular.
So consider the remaining case $k\leq s_{\alpha}-1$, $2r< \xi^{\alpha}_0$ and $\xi_{s_{\alpha}-1}^{\alpha}\neq\xi_{s_\alpha}^{\alpha}$. Again we consider several cases, where we start with $\xi^{\a}_{s_\a}\ge \frac{\xi^{\a}_{s_\a-1}}{2}$. Let $(p_i)_{i\in\frac{1}{2}\bz_+}\in \tilde\bs_{<s_{\alpha}}(r,s)$. If we prove that 
$p_{s_{\alpha}-\frac{1}{2}}\geq 2\xi^{\a}_{s_\a}-\xi^{\a}_{s_\a-1}+1$, then it is immediate that $\bold y^{\alpha}_{< s_{\alpha}}(r,s)w_\lambda\in U$. Indeed
\begin{align*}
\frac{1}{2}+2kr+\phi(\xi^{\alpha};k)\leq s&\leq \sum_{N\leq s_{\alpha}-1} \frac{2N+1}{2}p_{N+\frac{1}{2}} +\sum_{N\leq s_{\alpha}-1} 2Np_N
&\\&\leq (s_{\alpha}-\frac{1}{2})p_{s_{\alpha}-\frac{1}{2}}+2(s_{\alpha}-1)(r-\frac{1}{2}p_{s_{\alpha}-\frac{1}{2}})&\\&
\leq \frac{1}{2}p_{s_{\alpha}-\frac{1}{2}}+2r(s_{\alpha}-1).
\end{align*}
Thus we get
\begin{align*}
\frac{1}{2}p_{s_{\alpha}-\frac{1}{2}}&\ge \frac{1}{2}+2r(k-(s_{\alpha}-1))+\phi(\xi^{\alpha};k)&\\&
\ge \frac{1}{2}-2r((s_{\alpha}-1)-k)+\begin{cases}\sum_{s_\a>j\geq k+1} \xi^{\alpha}_{j}+\xi^\a_{s_\a}-\frac{1}{2}\xi^{\alpha}_{k+1},& \text{if $k<s_\a-1$}\\
\xi^\a_{s_\a}-\frac{1}{2}\xi^{\alpha}_{s_\a-1},& \text{if $k=s_\a-1$}
\end{cases}
&\\&
\ge \frac{1}{2}+\sum_{s_\a >j\geq k+1} (\xi^{\alpha}_{j}-2r)+\xi^\a_{s_\a}-\frac{1}{2}\xi^{\alpha}_{s_\a-1}\ge \frac{1}{2}+\xi^\a_{s_\a}-\frac{1}{2}\xi^{\alpha}_{s_\a-1}.
\end{align*}

We consider the remaining case $\xi^{\a}_{s_\a}<\frac{\xi^{\a}_{s_\a-1}}{2}.$ In this situation we will have two cases. First we consider 
$r\le \xi^\a_{s_\a}$. The earlier calculation shows that
$$\frac{1}{2}p_{s_{\alpha}-\frac{1}{2}}\ge \frac{1}{2}+2r(k-(s_{\alpha}-1))+\phi(\xi^{\alpha};k).$$
For $k=s_\a-1$ we get $\frac{1}{2}p_{s_{\alpha}-\frac{1}{2}}\ge \frac{1}{2}$ and for $k\le s_\a-2$ we get
$$\frac{1}{2}p_{s_{\alpha}-\frac{1}{2}}\ge \frac{1}{2}+\sum_{s_\a >j\geq k+2} (\xi^{\alpha}_{j}-2r)+(\frac{1}{2}\xi^{\alpha}_{s_\a-1}-r)+(\xi^\a_{s_\a}-r)\ge \frac{1}{2}.$$ Thus in either case $p_{s_{\alpha}-\frac{1}{2}}\ge 1.$ So we are done in this case.

Now let us consider the last case when $\xi^{\a}_{s_\a}<r<\frac{\xi^{\a}_{s_\a-1}}{2} $. If $p_{s_{\alpha}-\frac{1}{2}}\neq 0$ then we are done, so assume that $p_{s_{\alpha}-\frac{1}{2}}=0.$  If we prove that $p_{s_{\alpha}-1}\geq \xi^{\alpha}_{s_{\alpha}}+1$ then it follows that
$\bold y^{\alpha}_{< s_{\alpha}}(r,s)w_\lambda\in U$. Again using similar calculations we get,
\begin{align*}
\frac{1}{2}+2kr+\phi(\xi^{\alpha};k)\leq s&\leq \sum_{N\leq s_{\alpha}-2} \frac{2N+1}{2}p_{N+\frac{1}{2}} +\sum_{N\leq s_{\alpha}-1} 2Np_N
&\\&\leq 2(s_{\alpha}-1)p_{s_{\alpha}-1}+(2(s_{\alpha}-2)+1)(r-p_{s_{\alpha}-1})&\\&
\leq p_{s_{\alpha}-1}+r(2(s_{\alpha}-2)+1)
\end{align*}
which implies $p_{s_{\alpha}-1}\ge \frac{1}{2}+2r(k-(s_{\alpha}-1))+r+\phi(\xi^{\alpha};k).$
For $k=s_\a-1$ we obtain $p_{s_{\alpha}-1}\ge \frac{1}{2}+r\ge \frac{1}{2}+\xi^\a_{s_\a}$ and hence $p_{s_{\alpha}-1}\ge 1+\xi^\a_{s_\a}.$
For $k\le s_\a-2$ we get
\begin{align*}
p_{s_{\alpha}-1}&\ge \frac{1}{2}-2r(s_{\alpha}-k-2)+(\frac{1}{2}\xi^{\alpha}_{k+1}-r)+\sum_{s_\a>j\geq k+2} \xi^{\alpha}_{j}+\xi^\a_{s_\a}&\\&
\ge \frac{1}{2}+(\frac{1}{2}\xi^{\alpha}_{k+1}-r)+\sum_{s_\a >j\geq k+2} (\xi^{\alpha}_{j}-2r)+\xi^\a_{s_\a}\ge \frac{1}{2}+\xi^\a_{s_\a}
\end{align*}
Since $p_{s_{\alpha}-1}\in \bz_+$, we have $p_{s_{\alpha}-1}\ge 1+\xi^\a_{s_\a}$. This completes the proof.
\end{pf}

\subsection{}
Here we shall give special partitions, such that our corresponding module $V(\bxi)$ is isomorphic to a well-known graded representation for the hyperspecial twisted current algebra. Let $\widehat{V}(\Lambda)$ be the unique irreducible  highest weight  $\Gaff$-module with highest weight $\Lambda$. For each $w\in \widehat{W}$ the weight space $\widehat{V}(\Lambda)_{w(\Lambda)}$ is one-dimensional and the $\Baff$-module generated by $\widehat{V}(\Lambda)_{w(\Lambda)}$ is called the Demazure module, denoted by $\widehat{V}_w(\Lambda)$. 
Since the structure of the Demazure module depends only on $w(\Lambda)=-\lambda+\ell \Lambda_0+i\delta$ we shall write $\D(\ell,\lambda)[i]$ for $\widehat{V}_w(\Lambda)$. These modules  are  called level $\ell$ Demazure modules and are finite dimensional $\Baff$-modules. 
Although $\D(\ell,\lambda)[i]$ are by definition only $\Baff$-modules, for dominant $\lambda$ they admit a $\CG$--module structure which is independent if $i$ and hence for simplicity we denote them by $\D(\ell,\lambda)$. In this paper we will be concerned with $\CG$--stable Demazure modules.
More generally, one can associate a Demazure module with any element of $\widetilde{W}$ as follows: For $\Lambda\in \widehat{ P}^+$, $\sigma \in \Sigma$ and $w\in \widehat{W}$, define
$$ \widehat{V}_{w\sigma}(\Lambda)=\widehat{V}_w(\sigma \Lambda).$$

Demazure modules can also be presented as cyclic modules that have an explicit description of the annihilator of the generating element. The description of the Demazure modules for finite dimensional simple Lie algebras \cite[Theorem 3.4]{J85}, \cite[Proposition Fondamentale 2.1]{P89} was extended in \cite[Lemme 26]{M88} for Demazure modules associated to Kac-Moody Lie algebras. We record below the statement that is relevant for us for $\CG$--stable Demazure modules.
\begin{thm}\label{demq}
Let $\l\in P^+$. The Demazure module $\D(\ell,\l)$ is the graded $\CG$-module generated by an element $v_{\ell,\l}$ with the relations
\begin{align*}
x_\a^{k_\a+1} \cdot v_{\ell,\l}&=0,\ \a\in \widehat{R}_{\rm re}^+\cup R^-, ~k_\a=\max\{0,-(\l+\ell\Lambda_0,\a^\vee)\},&\\&
\CH_+ \cdot v_{\ell,\l}=0,\ \
h\cdot v_{\ell,\l}=\l(h)v_{\ell,\l},\ h\in\H.
\end{align*}
\hfill\qedsymbol
\end{thm}
For the sake of completeness we will write the exponents $~k_\a$ explicitly. We have 
\begin{align}\label{gleich} &k_{\alpha}=0 \ \text{ if } \alpha\in \widehat R_{\rm re}(+), ~k_{-\alpha+d_{\alpha}r\delta}=\max\{0,\lambda(\alpha^\vee)-\ell r\},\ \mbox{and} \\
\label{gleich2}&k_{-\frac{\alpha}{2}+(r+\frac{1}{2})\delta}=\max\{0,2\lambda(\alpha^\vee)-\ell(2r+1)\}.\end{align}
\subsection{}
For any $\alpha\in R^+$ we write 
$\lambda(\alpha^{\vee})=(s_{\alpha}-1)\ell+m_{\alpha},\ 0<m_{\alpha}\leq \ell.$
If $\lambda(\alpha^{\vee})=0$ set $s_{\alpha}=0=m_{\alpha}$. Denote by $\bxi(\ell,\lambda)=(\xi^{\alpha})_{\alpha\in R^+}$ be the $|R^+|$-tuple of partitions given by:
$\xi^{\alpha}$ is the empty partition if $\lambda(\alpha^{\vee})=0$ and otherwise, is the partition $(\ell^{s_{\alpha}},m_{\alpha})$. In the rest of this section we want to relate the modules $V(\bxi(\ell,\lambda))$ with the Demazure modules $\D(\ell,\lambda)$. In particular, we obtain the following theorem.
\begin{thm}\label{dem2}
Let $\ell \in \bn$ and $\lambda\in P^+$. We have an isomorphism of graded $\CG$ modules $V(\bxi(\ell,\lambda))\cong \D(\ell,\lambda)$. More precisely, the $\CG$ stable Demazure module is the quotient of $W_{\loc}(\lambda)$ by the submodule generated by the elements 
 \begin{align}\label{finer21}&\big\{ (x_{-\alpha+d_{\alpha}s_{\alpha}\delta})w_\lambda:\alpha\in R^+\big\}
\ \bigcup \big\{(x_{-\alpha+d_{\alpha}(s_{\alpha}-1)})^{m_{\alpha}+1}w_\lambda:\alpha\in R^+, \ m_{\alpha}<\ell\big\}\\
\label{finer22} &\big\{ (x_{-\frac{\alpha}{2}+(s_{\alpha}+\frac{1}{2})\delta})w_\lambda
:\alpha\in R_{\ell}^+\big\}
\bigcup
\big\{(x_{-\frac{\alpha}{2}+(s_{\alpha}-\frac{1}{2})\delta})^{(2m_{\alpha}-\ell)_{+}+1}w_\lambda:\alpha\in R_{\ell}^+, \ m_{\alpha}<\ell\big\}
\end{align} 
\begin{pf}
By Theorem~\ref{genmax2} the module $V(\bxi(\ell,\lambda))$ is the quotient of $W_{\loc}(\lambda)$ by the submodule generated by the elements \eqref{finer21} and \eqref{finer22}. Taking $r\in \{s_{\alpha},s_{\alpha}-1\}$ in \eqref{gleich} and \eqref{gleich2} we see with Theorem~\ref{demq} that the canonical map $W_{\loc}(\lambda)\longrightarrow \D(\ell,\lambda)$ factors through to a map of $\CG$-modules
$$V(\bxi(\ell,\lambda))\longrightarrow \D(\ell,\lambda).$$
To prove that this map is an isomorphism we must prove that the additional defining relations of $\D(\ell,\lambda)$ hold in $V(\bxi(\ell,\lambda))$. We obviously have $x_{\alpha}v_{\bxi(\ell,\lambda)}=0$ for all $\alpha\in \widehat R_{\rm re}(+)$ and by Corollary~\ref{cthird} we get $x_{-\alpha+d_{\alpha}r\delta}v_{\bxi(\ell,\lambda)}=x_{-\frac{1}{2}\alpha+(r+\frac{1}{2})\delta}v_{\bxi(\ell,\lambda)}=0$ if $r\geq s_{\alpha}$. So assume $r<s_{\alpha}$, which gives
\begin{align*}&
k_{-\alpha+d_{\alpha}r\delta}\geq (s_{\alpha}-1-r)\ell+m_{\alpha}\geq \sum_{j\geq r+1}\xi^{\alpha}_{j}&\\&
k_{-\frac{1}{2}\alpha+(r+\frac{1}{2})\delta}\geq \big(2(s_{\alpha}-1-r)\ell+2m_{\alpha}-\ell\big)_+\geq 2\phi(\xi^{\alpha};r)\end{align*}
Hence again by Corollary~\ref{cthird} we obtain a surjective $\CG$--module map $\D(\ell,\lambda)\longrightarrow V(\bxi(\ell,\lambda))$. This completes the proof.
\end{pf}
\end{thm}
As a corollary of the previous theorem we obtain that the level one Demazure modules are isomorphic to local Weyl modules, which was initially proven in \cite{CIK14}.
We can further simplify the presentation of Demazure modules if $\bxi(\ell,\lambda)$ is a tuple of rectangular partitions.
\begin{cor}\label{a2corf}
 Let $\bxi(\ell,\lambda)$ be a tuple of rectangular partitions. Then, $V(\bxi(\ell,\lambda))$ is isomorphic to the quotient of $W_{\loc}(\lambda)$ by the submodule generated by the elements
\begin{align}\label{nureinfacha2}&\big\{ (x_{-\alpha+d_{\alpha}s_{\alpha}\delta})w_\lambda,\  (x_{-\frac{\alpha_n}{2}+(s_{\alpha_n}+\frac{1}{2})\delta})w_\lambda:\alpha\in \Delta\big\}\end{align}
\begin{pf}
Let $\tilde{V}(\bxi(\ell,\lambda))$ be the quotient of $W_{\loc}(\lambda)$ by the submodule generated by \eqref{nureinfacha2} and $\tilde{v}_{\bxi(\ell,\lambda)}$ its cyclic generator. Note that it is enough to show that $$(x_{-\alpha+d_{\alpha}s_{\alpha}\delta})\tilde{v}_{\bxi(\ell,\lambda)}=0 ,\ (\mbox{resp. }(x_{-\frac{\alpha}{2}+(s_{\alpha}+\frac{1}{2})\delta})\tilde{v}_{\bxi(\ell,\lambda)}=0), \mbox{ for all $\alpha\in R^+$ (resp. $\alpha\in R_\ell^+$)}$$
Let $\alpha=\sum_ik_i\alpha_i$. We will proceed by induction on $\htt(\alpha)=\sum_ik_i$ (the height of $\alpha$). The initial step is obvious, so let $\alpha\in R^+-\Delta$ and write $\alpha=\beta+\gamma$ with $\htt(\beta)<\htt(\alpha),\htt(\gamma)<\htt(\alpha)$. Since $d_{\alpha}\alpha^{\vee}=d_{\beta}\beta^{\vee}+d_{\gamma}\gamma^{\vee}$ it follows that $d_{\alpha}s_{\alpha}=d_{\beta}s_{\beta}+d_{\gamma}s_{\gamma}$. The induction hypothesis gives
$$(x_{-\alpha+d_{\alpha}s_{\alpha}\delta})\tilde{v}_{\bxi(\ell,\lambda)}=\big[x_{-\beta+d_{\beta}s_{\beta}\delta},x_{-\gamma+d_{\gamma}s_{\gamma}\delta}\big]\tilde{v}_{\bxi(\ell,\lambda)}=0.$$
For a long root $\alpha=2(\alpha_i+\cdots+\alpha_{n})-\alpha_n$, set $\beta=\alpha_i+\cdots+\alpha_{n-1}$ and $\gamma=\alpha_n$. We obtain
$$(x_{-\frac{\alpha}{2}+(s_{\alpha}+\frac{1}{2})\delta})\tilde{v}_{\bxi(\ell,\lambda)}=\big[x_{-\beta+s_{\beta}\delta},x_{-\frac{\gamma}{2}+(s_{\gamma}+\frac{1}{2})\delta}\big]\tilde{v}_{\bxi(\ell,\lambda)}=0.$$
\end{pf}
\end{cor}

\section{The modules \texorpdfstring{$V(\bxi)$}{V} for the special twisted current algebras}\label{section4}
In this section, we consider the special twisted current algebras in the remaining
indecomposable twisted affine Lie algebras. The treatment of these is conceptually identical and technically easier than the considerations for the hyperspecial twisted current algebra. For completness we give a brief summary of their construction, where the details can be found in \cite{Ca05,FK11,K90}.

\subsection{} Let $\overline{\lie g}$ be a simple Lie algebra of type $\tt A_{2n-1}, n\geq 3$, $\tt D_{n+1}, n\geq 3$ or $\tt E_6$. Let $\mathcal{L}(\overline{\lie g}) =\overline{\lie g}\otimes\bc[t, t^{-1}]$ be the loop algebra of $\overline{\lie g}$ with the usual Lie bracket, given by the $\bc[t,t^{-1}]$--bilinear extension of the Lie bracket on $\overline{\lie g}$. The set of roots of $\overline{\lie g}$ is denoted by $R_{\overline{\lie g}}=R^+_{\overline{\lie g}}\cup R^-_{\overline{\lie g}}$ and the weight lattice (resp. the cone of dominant weights) is denoted by $P_{\overline{\lie g}}$ (resp. $P_{\overline{\lie g}}^+$). We denote by $V_{\overline{\lie g}}(\overline{\lambda})$ the irreducible finite--dimensional $\overline{\lie g}$--module with highest weight $\overline{\lambda}\in P_{\overline{\lie g}}^+$. Extend the graph autmorphism $\sigma$ to an order $m$ automorphism of $\mathcal{L}(\overline{\lie g})$ defined by $$\sigma(x\otimes t^i)=\xi^{-i} \sigma(x)\otimes t^i,$$
where $\xi$ is a primitive $m$--th root of unity.
Remark that the Lie  subalgebras $\overline{\lie g}[t]$,\ $\overline{\lie n}^{\pm}[t]$ and $\overline{\lie h}[t]$ are stabilized by $\sigma$.

The special twisted current algebra of an affine Lie algebra of type $\tt A_{2n-1}^{(2)}, n\geq 3$, $\tt D_{n+1}^{(2)}, n\geq 4$, $\tt E_6^{(2)}$, or $\tt D_4^{(3)}$ is isomorphic to the fixed point subalgebra $\CG$ of $\overline{\lie g}[t]$, for $\overline{\lie g}$  of type $\tt A_{2n-1}, n\geq 3$, $\tt D_{n+1}, n\geq 4$, $\tt E_6$ and $m=2$, and  of type $\tt D_4$ and $m=3$, respectively. Recall that both $\overline{\lie g}[t]$ and $\CG$ are naturally $\bz_+$--graded and the homogeneous components of degree zero are $\overline{\lie g}$ and, respectively, $\G$. 
With the notation from  Section~\ref{section1} we have that
$\CN^{\pm}$ (resp. $\CH$) is isomorphic to the fixed point subalgebra of $\overline{\lie n}^{\pm}[t]$ (resp. $\overline{\lie h}[t]$). Recall that we have triangular decomposition $$\CG=\CN^{-}\oplus\CH\oplus \CN^+.$$ 
\subsection{}
Similarly as we realized the hyperspecial twisted current algebra as a subalgebra of $\mathcal{L}(\mathfrak{sl}_{2n+1})$ (see Section~\ref{ev0}) we will realize $\CG$ as a subalgebra of $\mathcal{L}(\overline{\lie g})$. 
We fix a Chevalley basis $\{X^\pm_{\alpha}$, $H_i$ ~:~$i\in I, \alpha\in R^+_{\overline{\lie g}}\}$ for $\overline{\lie g}$. For any $\alpha\in R^+$, there is $\overline{\alpha}\in R^+_{\overline{\lie g}}$ such that $\overline{\alpha}\vert_{\lie h}=\alpha$.
Let $r\in \bz_+$ and $\alpha\in R^+$. Then,
\begin{align*}x_{\pm \alpha+d_{\alpha}r\delta}=\Big(\sum^{m-1}_{i=0}(\xi^i)^{d_{\alpha}r}X^{\pm}_{\sigma^i(\overline{\alpha})}\Big)\otimes t^{d_{\alpha}r}=X^{\pm}_{\alpha,d_{\alpha}r}\otimes t^{d_{\alpha}r},\ h_{\alpha,r\delta}=\Big(\sum^{m-1}_{i=0}(\xi^i)^rH_{\sigma^i(\overline{\alpha})}\Big)\otimes t^r\end{align*}
We remark that $\alpha_i^\vee=h_{i,0}$ for $i\in I$.
\begin{prop}\label{sl2}
For $\alpha\in R^+$, the subalgebra of $\CG$ generated by the elements 
$$\{x_{\pm\alpha+r\delta}~:~r\in\mathbb{Z}_+, \a+r\delta\in \widehat{R} \}$$ is isomorphic to $\lie{sl}_2[t]$.\hfill\qedsymbol
\end{prop}

\subsection{}
For a dominant integral $\lie g$ weight $\lambda$ the local Weyl module is generated by an element $w_{\lambda}$, with defining relations
$$\CN^{+}w_{\lambda}=0,\ \CH_+w_{\lambda}=0,\ h.w_{\lambda}=\lambda(h)w_{\lambda}, \mbox{ for $h\in \lie h$},\ x_{-\alpha}^{\lambda(\alpha^{\vee})}w_{\lambda}=0.$$
For a more general definition of local Weyl modules for equivariant map algebras we refer to \cite{FKKS11}.
For a $\lambda$-compatible tuple of partitions we define similarly a $\CG$--module $V(\bxi)$ as the graded quotient of $W_{\loc}(\lambda)$ by the submodule generated by the graded elements:
  \begin{equation*}
 \Big\{(x_{\alpha+d_{\alpha}\delta})^s(x_{-\alpha})^{s+r}w_\lambda : \alpha\in R^+, \ s,r\in\bn,  \   s+r\ge 1+ rk+\sum_{j\ge k+1}\xi^{\alpha}_j,\ \ {\rm{for\ some}}\ k\in\bn \Big\}.
\end{equation*}

The three presentations are clear from Proposition~\ref{sl2}, Lemma~\ref{garsl_2} and Proposition~\ref{max}, since all calculations are valid for $\mathfrak{sl}_2[t]$.

\subsection{}
Again we simplify the presentation of $V(\bxi)$ for rectangular or special fat hook partitions and find a connection to $\CG$--stable Demazure modules for very special types of partitions. For any $\alpha\in R^+$ we write again $\lambda(\alpha^{\vee})=(s_{\alpha}-1)\ell+m_{\alpha},\ 0<m_{\alpha}\leq \ell.$
Similarly, we denote by $\bxi(\ell,\lambda)=(\xi^{\alpha})_{\alpha\in R^+}$ be the $|R^+|$-tuple of partitions given by:
$\xi^{\alpha}$ is the empty partition if $\lambda(\alpha^{\vee})=0$ and otherwise, is the partition $(\ell^{s_{\alpha}-1},m_{\alpha})$. We collect our results in the next two theorems.

\begin{thm}
\label{genmax3} Let $\bxi=(\xi^{\alpha})_{\alpha\in R^+}$ be a $\lambda$--compatible $|R^+|$--tuple of partitions. Assume  that  $\xi^\alpha $ is  either rectangular or special for $\alpha\in R^+$. Then, $V(\bxi)$ is isomorphic to the quotient of $W_{\loc}(\lambda)$ by the submodule generated by the elements
\begin{align}\label{finer3}&\big\{ (x_{-\alpha+d_{\alpha}s_{\alpha}\delta})w_\lambda:\alpha\in R^+\big\}
\ \bigcup \big\{(x_{-\alpha+d_{\alpha}(s_{\alpha}-1)\delta})^{\xi^\alpha_{s_\alpha}+1}w_\lambda:\alpha\in R^+, \ \ \xi^{\alpha} \ {\rm{{special}}}\big\}
\end{align} \vskip 6pt

\begin{pf}
The proof proceeds similarly to the first part of the proof of Theorem~\ref{genmax2}.
\end{pf}
\end{thm}
\subsection{}
As before, for any pair $(\ell,\lambda)$ there exists an indecomposable $\CG$--stable Demazure module $\D(\ell,\lambda)$, which can be presented as a cyclic module that has an explicit description
of the annihilator of the generating element. For more details we refer to \cite[Lemma 26]{M88} and \cite[Section 4]{FK11}.

\begin{thm}\label{dem3}
The module $V(\bxi(\ell,\lambda))$ is isomorphic to the Demazure module $\D(\ell,\lambda)$ as a $\CG$--module.
More precisely, the $\CG$--stable Demazure module is the quotient of $W_{\loc}(\lambda)$ by the submodule generated by the elements 
 \begin{align}\label{finer23}&\big\{ (x_{-\alpha+d_{\alpha}s_{\alpha}\delta})w_\lambda:\alpha\in R^+\big\}
\ \bigcup \big\{(x_{-\alpha+d_{\alpha}(s_{\alpha}-1)})^{m_{\alpha}+1}w_\lambda:\alpha\in R^+, \ m_{\alpha}<\ell\big\}
\end{align} 
\begin{pf}
Using the simplified presentation in Theorem~\ref{genmax3} the proof proceeds similarly as the proof of Theorem~\ref{dem2}.
\end{pf}
\end{thm}
Again one can see with the previous theorem that the level one Demazure modules are
isomorphic to local Weyl modules, which was initially proven in \cite{FK11} for the special twisted current algebras.
Similarly to Corollary~\ref{a2corf} we can further simplify the presentation of Demazure modules if $\bxi(\ell,\lambda)$ is a tuple of rectangular partitions.
\begin{cor}\label{restcorf}
 Let $\bxi(\ell,\lambda)$ be a tuple of rectangular partitions. Then, $V(\bxi(\ell,\lambda))$ is isomorphic to the quotient of $W_{\loc}(\lambda)$ by the submodule generated by the elements
\begin{align}\label{nureinfachrest}&\big\{ (x_{-\alpha+d_{\alpha}s_{\alpha}\delta})w_\lambda:\alpha\in \Delta\big\}\end{align}\hfill\qedsymbol
\end{cor}
\section{Tensor product decomposition of twisted Demazure modules}\label{section5}

\subsection{} In this section, we give a tensor product decomposition of twisted Demazure modules.  The main result of this section is the following. 
\begin{thm}\label{tensordec}
Let $\lambda\in P^+$ and $p, \ell\in \bn$ and write $$\lambda=\ell \Big(\sum\limits_{i=1}^{p}\lambda_i\Big)+\lambda_0,\ \text{for some } \lambda_i\in P^+,\ 0\le i\le p.$$
If $\widehat{\lie g}$ is of type $\tt E^{(2)}_6$ we assume that $\ell\Lambda_0-w_0\lambda_0\in \widehat{P}^+$.
Then we have an isomorphism of $\lie g$--modules
$$D(\ell,\lambda)\cong_{\lie g} D(\ell,\ell \lambda_1)\otimes \cdots \otimes D(\ell, \ell \lambda_p)\otimes D(\ell, \lambda_0).$$
\end{thm}

\subsection{} We make a few remarks before proceeding to the proof. The tensor product decomposition of $\lie g$--stable Demazure modules was proved in \cite{FoL06} for the special case $\lambda_0=0$.
A more general case is considered recently in \cite{CSVW14} for the untwisted affine algebras with the exceptions $\tt E^{(1)}_{6,7,8}$ and $\tt F^{(1)}_4$.
As in these papers, the proof of Theorem~\ref{tensordec} uses the theory of Demazure operators and the following keyfact. 
\begin{lem}\label{tensordeclemma}
 Assume that $\widehat{\lie g}$ is not of type $\tt E^{(2)}_6$. Let $\lambda\in P^+$ and $\ell\in \bn$ such that $\lambda(\a^\vee)\leq \ell$ for all $\alpha\in \Delta$. Then there exists $\mu\in P^+$ and $w\in W$ such that
 $wt_{\mu}(-\lambda+\ell \Lambda_0)\in \widehat P^+.$ 
\end{lem}
\begin{proof}
If $\widehat{\lie g}$ is not of type $\tt A^{(2)}_{2n}$ we have $\alpha_0^\vee=K-\theta^\vee$ (as in the untwisted cases), where $K$ denotes the central element of $\widehat{\lie g}$. Hence, the proof of the lemma proceeds similarly to \cite[Proposition 3.5]{CSVW14}. Now we focus on type $\tt A^{(2)}_{2n}$, where we have $\alpha_0^\vee=K-2\theta^\vee$. The proof of the lemma relies on the following fact: we claim that there exists $\mu\in P^+$ such that
 $$ |(\ell\mu-\lambda)(2\alpha^\vee)|\le \ell \ \ \text{for all $\alpha\in R_{\ell}^+$.}$$ We assume the claim and complete the proof of the lemma. Since $\ell\mu-\lambda\in P,$ there exists $w\in W$ such that $w(\ell\mu-\lambda)\in P^+.$ 
A simple calculation shows that  $wt_\mu(-\lambda+\ell\Lambda_0)=\ell\Lambda_0+w(\ell\mu-\lambda) \mod \bz\delta.$ Since $w(\ell\mu-\lambda)\in P^+$, 
$$\ell\Lambda_0+w(\ell\mu-\lambda)\in \widehat{P}^+ \ \  \text{iff} \ \ 
\ell\Lambda_0+w(\ell\mu-\lambda)(\alpha_0^\vee)=\ell-w(\ell\mu-\lambda)(2\theta^\vee)\ge 0.$$
Now $\ell-w(\ell\mu-\lambda)(2\theta^\vee)=\ell-(\ell\mu-\lambda)(2\alpha^\vee)$ for some $\alpha\in R_{\ell},$ which is non-negative by our claim. It remains to prove the claim which will proceed by induction on $n$. 
Observe that $\{\alpha_i^\vee+\cdots +\alpha_n^\vee : 1\le i\le n\}$ is the set of all positive long coroots of $\tt C_n.$ 
The induction begins at $n=1$ where we can take $\mu=0$ if $2\lambda(\alpha^{\vee})\leq \ell$ and otherwise $\mu=\omega_1$.
For the inductive step assume  that the result is proved for the  $\tt C_{n-1}$--subdiagram of $\tt C_n$ defined by the  simple roots $\{\alpha_2,\dots\alpha_n\}$ of $\tt C_n$. Let  $\mu'=\sum_{j=2}^{n}s_j\omega_j\in P^+$  such that 
$$|(\ell\mu'-\lambda)(2\alpha^\vee)|\le \ell,$$ for all positive long roots $\alpha$ of $\tt C_{n-1}$.
 The only additional positive long root in $\tt C_n$ is the highest coroot $\theta^\vee$. Moreover,  $\theta^\vee-\alpha_1^\vee$ is a coroot of $\tt C_{n-1}$ and so
  we take   $$\mu=\begin{cases}\mu'&   \text{if } |\lambda(2\theta^\vee)-(\ell\mu')(2\theta^\vee-2\alpha_1^\vee)|\le \ell,\\
\; \omega_1+\mu'& {\rm{otherwise}}.
\end{cases}$$ A simple calculation completes the proof.
\end{proof}

\subsection{} We now recall the Demazure character formula from \cite[Chapter VIII]{KU02}. Denote by $\D_w$ the Demazure operator associated with an arbitrary element $w\in \widetilde{W}$. 
Then we have 
$$\charc_{\widehat{\lie h}}\ V_{w\sigma}(\Lambda)=\charc_{\widehat{\lie h}}\ V_w(\sigma(\Lambda))=\D_{w}(e(\sigma\Lambda))=\D_{w\sigma}(e(\Lambda)),$$
where $\charc_{\widehat{\lie h}}$ is the character function with respect to $\widehat{\lie h}$ defined in the obvious way. 
We note here that we are only interested in $\lie g$--module structure of $\D(\ell,\lambda)$, and so in particular that we are interested only in $\lie h$--characters and hence it is enough to calculate the Demazure characters modulo the ideal $I_\delta$ generated by $e(\delta)-1.$
The next lemma is proven similarly as \cite[Proposition 2.8]{CSVW14}.
\begin{lem}\label{length}
Let $\lambda, \mu\in P^+$ and $w\in W$. Then we have
$\ell(t_{-\mu}t_{-\lambda}w)=\ell(t_{-\mu})+\ell(t_{-\lambda}w)$, where $\ell(-)$ denotes the extended length function of $\widetilde{W}$. 
In particular, $\D_{t_{-\mu}t_{-\lambda}w}=\D_{t_{-\mu}}\D_{t_{-\lambda}w}$.
\end{lem}

\subsection{Proof of Theorem~\ref{tensordec}} 
We assume that $\lambda_0(\a^\vee)\leq \ell$ for all $\alpha\in \Delta$. The general case can be easily deduced. 
By our assumption and Lemma~\ref{tensordeclemma} there exists $\mu\in P^+, w\in W$ and $\Lambda\in \widehat{P}^+$ such that
$t_{-\mu}w(\Lambda)=-\lambda_0+\ell\Lambda_0.$
Then we have $e(\ell\Lambda_0)\charc_{\lie h} \D(\ell,\lambda_0)=\D_{t_{-\mu}w}(e(\Lambda)) \mod  I_\delta.$
Since $$t_{-\sum_{i=1}^{p}\lambda_i}t_{-\mu}w(\Lambda)=-\left(\ell\sum\limits_{i=1}^{p}\lambda_i+\lambda_0\right)+\ell\Lambda_0 \mod \ I_\delta,$$
we have, by definition, $V_{t_{-\sum_{i=1}^{p}\lambda_i}t_{-\mu}w}(\Lambda)=\D(\ell, \lambda)$ and again we have
$$e(\ell\Lambda_0)\charc_{\lie h}\D(\ell,\lambda)= \D_{t_{-\sum_{i=1}^{p}\lambda_i}t_{-\mu}w}(e(\Lambda)) \mod  I_\delta.$$
From Lemma \ref{length} we get $\ell\left(t_{-\sum_{i=1}^{p}\lambda_i}t_{-\mu}w\right)=\sum\limits_{i=1}^{p}\ell(t_{-\lambda_i})+\ell(t_{-\mu}w),$ and hence using the properties of the Demazure operators we get,
$$\D_{t_{-\sum_{i=1}^{k}\lambda_i}t_{-\mu}w}\big(e(\Lambda)\big)=\D_{t_{-\lambda_1}}\cdots \D_{t_{-\lambda_k}}\D_{t_{-\mu}w}(e(\Lambda)) \mod I_\delta.$$
Then $$e(\ell\Lambda_0)\charc_{\lie h}\D(\ell,\lambda)=\D_{t_{-\lambda_1}}\cdots \D_{t_{-\lambda_k}}(e(\ell\Lambda_0)\charc_{\lie h} \D(\ell,\lambda_0)) \mod I_\delta.$$
Since $\charc_{\lie h} \D(\ell,\lambda_0)$ is $\widetilde{W}$--invariant, we get (see \cite[Lemma 13]{FoL06})
 $$e(\ell\Lambda_0)\charc_{\lie h}\D(\ell,\lambda)=\D_{t_{-\lambda_1}}\cdots \D_{t_{-\lambda_k}}(e(\ell\Lambda_0))\charc_{\lie h} \D(\ell,\lambda_0) \mod I_\delta.$$
Now since $\D_{t_{-\lambda_i}}(e(\ell\Lambda_0))=e(\ell\Lambda_0)\charc_{\lie h}\D(\ell,\ell\lambda_i)$ for all $1\le i\le p$, we get by repeated use of earlier arguments that
$$e(\ell\Lambda_0)\charc_{\lie h}\D(\ell,\lambda)=e(\ell\Lambda_0)\charc_{\lie h}\D(\ell,\ell\lambda_1)\cdots \charc_{\lie h}\D(\ell,\ell\lambda_k)\charc_{\lie h}\D(\ell,\lambda_0) \mod \ I_\delta.$$
The theorem is now immediate since $\bz[P]\hookrightarrow \bz[\widehat{P}]/I_\delta$ is an inclusion.
This completes the proof.
\subsection{}
We conclude this section by proving the invertibility of certain matrices which will be needed in the proof of Theorem~\ref{mainthmsection6}.
Let $r,s,p\in \bn,\  d,c\in \bz_+$ and write $$r=\sum^{s-1}_{j=0}\sum^p_{l=1}r_{jp+l}.$$
Let $A$ be the $(r\times r)$ square matrix given as in one of the following cases. 
\vspace{0,2cm}

\textsc{Case 1:}

\vspace{0,2cm}
Let $s=1, m=2$ and write $i=r_1+\cdots+r_{q}+y$ for some $0\leq q<p$ and $1\leq y\leq r_{q+1}$, then
\begin{equation}\label{m1}a_{i,j}(d)=\binom{2(d+j)-3}{y-1}z_{q+1}^{2(d-1+j)-y}\end{equation}
where we understand $a_{i,1}(0)=(-1)^{y-1}z_{q+1}^{-y}$.
\vspace{0,2cm}


\begin{rem}\label{remmat}
Because of 
$$\sum^N_{k=y}(-1)^k
\binom{N}{k}\binom{k-1}{y-1}=(-1)^{y}\ \mbox{ for all } N,y\in \bn,\ y\leq N$$ the entries $a_{i,1}(0)$ in \eqref{m1} can be rewritten as $a_{i,1}=[\varphi_{z_{q+1}}(t^{-1})]_{y-1}$, where $[f(t)]_u$ denotes the coefficient of $t^u$ in $f(t)$.
\end{rem}





\textsc{Case 2:}

\vspace{0,2cm}
Let $s=m$ and write 
$i=r_1+\cdots+r_{kp+q}+y$ for some $0\leq k< m$, $0\leq q<p$ and $1\leq y\leq r_{kp+(q+1)}$, then
\begin{equation}\label{m2}a_{i,j}(d)=\xi^{k(d+j-1+c)}\binom{d+j-1}{y-1}z_{q+1}^{d+j-y}\end{equation}

\begin{lem}\label{invmat}
The matrices in \eqref{m1} and \eqref{m2} are invertible.
\begin{pf}
The proof in both cases is similar and we focus ourself on Case 1. Let $a_{0},\dots,a_{r-1}$ be the coefficients, such that the corresponding linear combination of the columns of $A$ is zero. Let $Q(x)=
a_{0}+a_1x^2+a_2x^4+\cdots a_{r-1}x^{2(r-1)}$. By the definition of the matrix we obtain that $Q(x)$ has the following properties
$$Q(\pm z_i)=0,\ \frac{Q^{(j)}(\pm z_i)}{j!}=0 ,\ 1\leq i \leq p,\ 1\leq j\leq r_i-1.$$
Since  $z_i^2\neq z_j^2$ for all $i\neq j$ we obtain that $Q(x)$ is divisible by $\prod^p_{i=1}(x^2-z^2_i)^{r_i}$, which is a contradiction to the degree of $Q(x)$. Hence $a_{0}=\cdots=a_{r-1}=0$.

\end{pf}
\end{lem}


\section{Fusion product decomposition of twisted Demazure modules}\label{section6}
We shall give a fusion product decomposition of twisted Demazure modules analogous to the untwisted case (see for instance \cite{CSVW14},\cite{FoL07}). For the definition of fusion products for untwisted modules we refer to \cite{FL99}. In this section we do not consider the hyperspecial twisted current algebra seperately, i.e. $\CG$ stands for any twisted current algebra. Recall that $\CG$ can be realized as a subalgebra of $\mathcal{L}(\overline{\lie g})$, where $\overline{\lie g}=\mathfrak{sl}_{2n+1}$ if $\CG$ is hyperspecial (see Section~\ref{ev0}) and else as in Section~\ref{section4}. Even when $\CG$ is the special twisted current algebra, the main problem of defining fusion products is that $\CG$ is not stabilized by the Lie algebra homorphism $\overline{\lie g}[t]\longrightarrow \overline{\lie g}[t], t^k\mapsto (t+a)^k, a\in \bc^{\times}$. One needs different methods for the definition. We shall use freely the notation established in the earlier sections without further comments.

\subsection{}Before we define fusion products and state our main theorem we recall a general construction from \cite{FL99}. The element $d$ defines a $\mathbb Z_+$--graded Lie algebra structure on $\CG$. Let $\bu(\CG)[k]$ be the homogeneous component of degree $k$ (with respect to the grading induced by $d$) and recall that it is  a $\lie g$--module for all $k\in\mathbb{Z}_+$. Suppose now that we are given a  $\CG$--module $V$ which is generated by $v$. Define an increasing filtration $0\subset V^0\subset V^1\subset\cdots $ of $\lie g$-submodules of $V$ by$$V^k=\bigoplus_{s=0}^k \bu(\CG)[s] v.$$ The associated graded vector space $\gr V$ admits an action of $\CG$ given by: 
$$
x(v+V^{k})= xv+ V^{k+s},\ \ x\in\CG[s],\ \ v\in V^{k+1}.
$$
Furthermore, $\gr V$ is a cyclic $\CG$--module with cyclic generator $\bar v$, the image of $v$ in $\gr V$.

\subsection{}
For the rest of this section we shall relate the finite dimensional representation theory of $\CG$ and  $\mathcal{L}(\overline{\lie g})$. Any ideal of finite codimension in $\mathcal{L}(\overline{\lie g})$ is of the form $\overline{\lie g}\otimes \mathfrak{I}$ for some ideal $\mathfrak{I}\subset \bc[t,t^{-1}]$ and
$$
\mathfrak{I}\supseteq \big((t-z_1)\cdots (t-z_k)\big)^N \bc[t,t^{-1}]
$$
for some non-zero distinct complex numbers $z_s$, $1\le s\le k$ and $N\in\bn$. 
For $z\in \bc^{\times}$ and $N\in \bn$, we set $\overline{\lie g}_{z,N}=\overline{\lie g}\otimes \frac{\bc[t,t^{-1}]}{(t-z)^N}$ and for $\boz= (z_1,\dots, z_k)$ we set $\overline{\lie g}_{\boz,N}=\oplus^k_{s=1}\overline{\lie g}_{z_s,N}$. We see by the Chinese remainder theorem that  any finite--dimensional module $ \overline V$ of $\mathcal{L}(\overline{\lie g})$  can be regarded as a module for $\overline{\lie g}_{\boz,N}$ for some $\boz= (z_1,\dots, z_k)\in(\bc^\times)^k$ with pairwise distinct entries and some $N\in\bn$. Conversely, given a module $V$  of $\overline{\lie g}_{\boz,N}$ we shall construct a $\mathcal{L}(\overline{\lie g})$ and $\CG$--module respectively.

Given $\boz= (z_1,\dots, z_k)\in(\bc^\times)^k$ and $N\geq 1$, let 
$$
\ev
_{\boz, N}: L(\overline{\lie g})\to \overline{\lie g}_{\boz,N},
$$
be the canonical Lie algebra morphism and let 
$
\Psi_{\boz, N}
$ be the restriction of $\ev_{\boz, N}$ to the current algebra $\CG$. 
Hence we get modules $\ev^*_{\boz,N}V$ of $\mathcal{L}(\overline{\lie g})$ and $\Psi_{\boz, N} ^*V$ of $\CG$ by pulling back $V$ through the morphisms $\ev_{\boz, N}$ respectively $\Psi_{\boz, N}$. The proof of the following lemma can be found in \cite{CIK14}.

\begin{lem}\label{surm} If $z_s\ne z_p$ for $1\le s\ne p\le k$ then $\ev_{\boz, N}$ is surjective. If  $z_s^m \ne  z_p^m$ for $1\le s\ne p\le k$, then the restriction of $\Psi_{\boz, N}$ to $\oplus_{s\geq r}\CG[s]$ is surjective for any $r\in \bz_+$.\hfill\qedsymbol
\end{lem} 
So if in addition $V$ is cyclic, we obtain under further restrictions on $\boz$ that $\ev^*_{\boz,N}V$ respectively $\Psi_{\boz, N}^*V$ is a cyclic $\mathcal{L}(\overline{\lie g})$ respectively $\CG$--module. 
\subsection{}
We take the next proposition as a starting point for the definition of fusion products of $\CG$--modules.
\begin{prop}\label{cmf}
Let $V_1,\dots,V_p$ be cyclic finite--dimensional $\mathcal{L}(\overline{\lie g})$ modules and $N\in \bn$, such that the action of $\mathcal{L}(\overline{\lie g})$ on $V_i$ factors through $\overline{\lie g}_{\boz_i,N}$ for $i=1,\dots,p$ and $\boz_i=(z_{(1,i)}\dots, z_{(k_i,i)})$. 
Suppose $z_{(r,i)}^m\neq z_{(s,j)}^m$ for all $(r,i)\neq (s,j)$. Furthermore, let $W$ be a cyclic finite--dimensional graded $\CG$--module. Then we have that
$$\Psi^{*}_{\boz_1,N}V_1\otimes \cdots\otimes \Psi^{*}_{\boz_p,N}V_p\otimes W$$
is a cyclic $\bu(\CG)$ module.
\begin{pf}
Let $k=\sum^p_{i=1}k_i$ and $\boz= (z_{(1,1)}\dots, z_{(k_1,1)},z_{(1,2)}\dots, z_{(k_2,2)},\dots)\in(\bc^\times)^k$. The tensor product $V_1\otimes \cdots\otimes V_p$ is a cyclic module for $\oplus_{i=1}^p \overline{\lie g}_{\boz_i,N}$ and since $\Psi_{\boz, N}$ is surjective by Lemma~\ref{surm} we obtain that $\Psi^{*}_{\boz, N}(V_1\otimes \cdots\otimes V_p)\cong \Psi^{*}_{\boz_1,N}V_1\otimes \cdots\otimes \Psi^{*}_{\boz_p,N}V_p$ is a cyclic module for $\bu(\CG)$. In fact, more generally we see that $$\Psi^{*}_{\boz, N}(V_1\otimes \cdots\otimes V_p)=\bu\big(\oplus_{s\geq r}\CG[s]\big)v_1\otimes\cdots\otimes v_p$$
for any $r\in\bz_+$. Since $W$ is finite--dimensional graded $\CG$--module it follows that 
$\CG[r]W=0,$
for $r$ sufficiently large. Hence, $\Psi^{*}_{\boz_1,N}V_1\otimes \cdots\otimes \Psi^{*}_{\boz_p,N}V_p\otimes W$
is a cyclic $\bu(\CG)$ module.
\end{pf}
\end{prop}

\begin{defn}
Let $V_1,\dots,V_p,W$ as in Proposition~\ref{cmf}. 
We define the fusion product of $V_1,\dots,V_p,W$ as 
$$V_1*\cdots*V_p*W:=\gr \big(\Psi^{*}_{\boz_1,N}V_1\otimes \cdots\otimes \Psi^{*}_{\boz_p,N}V_p\otimes W\big)$$
\end{defn}
Note that the definition of the fusion product depends as in the untwisted case on parameters, so it would be more appropriate to denote the fusion product by $(V_1*\cdots*V_s*W)(\mathbf z,N)$. To keep the notation as simple as possible we omit almost always the parameters in the notation for the fusion product. It is conjectured that the fusion product is in fact independent of $\mathbf z$ and is proved in certain cases by various people (see for instance \cite{CSVW14,FF02,FL99,FoL07}). We will prove the independence of $\mathbf z$ for twisted Demazure modules.
\subsection{}
Our aim is to write $\D(\ell,\lambda)$, $\lambda\in P^+$, as a fusion product of suitable $\bu(\CG)$--modules. In particular, we will relate a twisted module with a untwisted module and therefore we shall regard $\lie h^*$ as a subspace of $(\overline{\lie{h}})^*$ by extending $\mu\in\lie h^*$ as follows
$$\mu(H_i)=\mu(\alpha_i^{\vee}),\ \ 1\le i\le n,\qquad  \mu(H_i)=0,\ \ i>n.$$ 
Via this identification, the set of fundamental weights for $\overline{\lie g}$ contains the fundamental weights for $\lie g$, allowing us to denote by $\omega_i$ ($i\in I$), the fundamental weights  for $\overline{\lie g}$.
Conversely, given $\mu\in (\overline{\lie{h}})^*$ we regard $\mu$ as an element of $\lie h^*$ by restricting $\mu$ to $\lie h$. So we have a map $$(\overline{\lie{h}})^*\rightarrow \lie h^{*},\ \mu\mapsto \mu\vert \lie h.$$ 
Note that the image of $P^{+}_{\overline{\lie g}}$ under the above map is contained in $P^+$ and we denote the set of all preimages of $\lambda\in P^+$ by $\lambda^{\sigma}$. In other words $\lambda^{\sigma}=\{\overline{\lambda}\in P^{+}_{\overline{\lie g}}\mid \overline{\lambda}\vert\lie h=\lambda\}.$ Especially $\lambda\in \lambda^{\sigma}$.

\subsection{}
Any untwisted $\overline{\lie g}$-stable Demazure module $\D_{\overline{\lie g}}(\ell,\lambda)$ is a finite--dimensional graded module for $\overline{\lie g}[t]$ and hence cyclic for $\overline{\lie g}\otimes \frac{\bc[t]}{t^N}$ for some $N\in \mathbb N.$
Since we have an isomorphism 
$$\varphi_z: \frac{\bc[t,t^{-1}]}{(t-z)^N}\longrightarrow \frac{\bc[t]}{t^N},\quad \big(\mbox{resp.}\ \varphi_z(\overline{\lie g}): \overline{\lie g}_{z,N}\longrightarrow \overline{\lie g}\otimes\frac{\bc[t]}{t^N}\big),\ z\in \bc^{\times}$$ 
the module $\D_{\overline{\lie g}}(\ell,\lambda)$ can be regarded as a cyclic module for $\overline{\lie g}_{z,N}$ by pulling back $\D_{\overline{\lie g}}(\ell,\lambda)$
through $\varphi_z(\overline{\lie g})$. To avoid introducing more notation we will use $\D_{\overline{\lie g}}(\ell,\lambda)$ to denote the above representation too.
Therefore, $\Psi^{*}_{z,N}\D_{\overline{\lie g}}(\ell,\lambda)$ is well-defined and cyclic as a $\CG$--module.
For the rest of this section we fix $\mathbf z=(z_1,\dots,z_p)$, such that $z_i^m\neq z_j^m$ for all $i\neq j$.
Note that 
$$\varphi_{z}(t^{k})=(t+z)^{k} \mbox{ if }k\geq 0 \mbox{ and }\varphi_{z}(t^{-1})=\sum_{j=1}^{N}(-1)^{j-1}\binom{N}{j}(t+z)^{j-1}z^{-j}.$$

We record the following lemma whose proof is a combination of the results in
\cite{HKOTT02} and \cite{H06} and an induction argument. A brief sketch of the proof is postponed to the end of Section~\ref{section7}. 
\begin{lem}\label{lemwe}
Let $\overline{\lambda}\in \lambda^{\sigma}$, then
$$ \dim \D_{\overline{\lie g}}(\ell,\ell\overline{\lambda})=\dim \D(\ell,\ell\lambda)$$

\end{lem}

\subsection{}
Recall that the presentation of $\D(\ell,\lambda)$ is greatly simplified in Theorem~\ref{dem2} and Theorem~\ref{dem3} respectively. Our main theorem of this section is the following. 

\begin{thm}\label{mainthmsection6}
Let $\lambda\in P^+$ and $p,\ell \in \bn$ and write
$$\lambda=\ell \Bigg(\sum^{p}_{i=1}\lambda_i\Bigg)+\lambda_{0},\ \text{for some } \lambda_i\in P^+,\ 0\leq i\leq p.$$ 
If $\widehat{\lie g}$ is of type $\tt E^{(2)}_6$ we assume that $\ell\Lambda_0-w_0\lambda_0\in \widehat{P}^+$.
Fix arbitrary elements $\overline{\lambda}_i\in \lambda_i^{\sigma}$, $1\leq i\leq p$. Then we have an isomorphism of $\CG$--modules

$$\D(\ell,\lambda)\cong_{\CG}  \D_{\overline{\lie g}}(\ell,\ell\overline{\lambda}_1)*\cdots* \D_{\overline{\lie g}}(\ell,\ell\overline{\lambda}_{p})*\D(\ell,\lambda_{0}).$$
\begin{pf}
Recall that the RHS is a cyclic $\CG$--module generated by $\mathbf v=v_{\ell,\ell\overline{\lambda}_1}* \cdots*v_{\ell,\ell\overline{\lambda}_{p}}*v_{\ell,\lambda_{0}}$ (see Proposition~\ref{cmf}) and since the dimension of both sides coincide by the tensor product decomposition (see Theorem~\ref{tensordec}) and Lemma~\ref{lemwe} we only need to show that we have a surjective map from $\D(\ell,\lambda)$ to the fusion product. Since the RHS is obviously a quotient of the local Weyl module $W_{\loc}(\lambda)$, we shall verify that $\mathbf v$ satisfies the simplified relations in \eqref{finer1}-\eqref{finer2} and \eqref{finer3} respectively. 
We start proving the theorem for the hyperspecial twisted current algebra.
First consider the elements $x_{-\alpha+2r\delta}$ for $r\in \bz_+$ and $\alpha=2(\alpha_i+\cdots+\alpha_n)-\alpha_n$. For any choice of complex numbers $c_0,\dots,c_{r-1}$ we obtain
\begin{align*} x_{-\alpha+2r\delta}\mathbf v&=\Big(x_{-\alpha+2r\delta}-\sum^{r-1}_{k=0} c_k x_{-\alpha+2k\delta}\Big)\mathbf v
&\\&=\sum^p_{q=1}v_{\ell,\ell\overline{\lambda}_{1}}*\cdots*\Big(X_{i,2n+1-i}^- \otimes \big((t+z_q)^{2r- 1}-\sum^{r-1}_{k=0} c_k\varphi_{z_q}(t^{2k- 1})\big)\Big)v_{\ell,\ell\overline{\lambda}_{q}}*\cdots*v_{\ell,\lambda_{0}}&\\&+\Big(v_{\ell,\ell\overline{\lambda}_{1}}*\cdots*v_{\ell,\ell\overline{\lambda}_{p}}*\big(x_{-\alpha+2r\delta}-\sum^{r-1}_{k=0} c_k x_{-\alpha+2k\delta}\big)v_{\ell,\lambda_{0}}\Big)\end{align*}
We shall make a particular choice of the coefficients $c_k$ such that for all $1\leq q\leq p$
\begin{equation}\label{wieesop}\Big(X_{i,2n+1-i}^- \otimes \big((t+z_q)^{2r- 1}-\sum^{r-1}_{k=0} c_k\varphi_{z_q}(t^{2k- 1})\big)\Big)v_{\ell,\ell\overline{\lambda}_{q}}=X_{i,2n+1-i}^- \otimes t^{\lambda_q(\alpha^{\vee})}v_{\ell,\ell\overline{\lambda}_{q}} \end{equation}
Write 
\begin{equation}\label{write}\lambda(\alpha^{\vee})=(s_{\alpha}-1)\ell+m_{\alpha},\quad \lambda_{0}(\alpha^{\vee})=(d-1)\ell+m_{\alpha},\quad d\in \bz_+\end{equation}
Let $r=s_{\alpha}=d+\big(\lambda_1+\cdots+\lambda_p\big)(\alpha^{\vee})$ and set $c_0,\dots,c_{d-1}=0$. We can transform \eqref{wieesop} into a system of linear equations ($(r-d)$ equations with $(r-d)$ unknowns). The corresponding square matrix $A=(a_{i,j}(d))$ is exactly the matrix given in \eqref{m1} ($r_q=\lambda_q(\alpha^{\vee}$)), which is invertible by Lemma~\ref{invmat}. Hence we can choose coefficients $c_0,\dots,c_{r-1}$ such that \eqref{wieesop} holds. Since 
$$\lambda_q(\alpha^{\vee})=\overline{\lambda}_q(\alpha^{\vee})=\overline{\lambda}_q(H_{i,2n+1-i}) \mbox{ and } x_{-\alpha+k\delta}v_{\ell,\lambda_{0}}=0 \mbox{ for all $k\geq d$,}$$ we are done in this case.
Now suppose $m_{\alpha}<\ell$ and set $r=s_{\alpha}-1=(d-1)+\big(\lambda_1+\cdots+\lambda_p\big)(\alpha^{\vee})$ and $c_0,\dots,c_{d-2}=0$. Then we can transform \eqref{wieesop} again into a system of linear equations ($(r-d+1)$ equations with $(r-d+1)$ unknowns), where the corresponding square matrix is of similar form.
Hence \eqref{wieesop} holds and by repeating this argument we get
$$x_{-\alpha+2r\delta}^{m_{\alpha}+1}\mathbf v=c_{d-1}^{m_{\alpha}+1}\big(v_{\ell,\ell\overline{\lambda}_{1}}*\cdots*v_{\ell,\ell\overline{\lambda}_{p}}*x_{-\alpha+2(d-1)\delta}^{m_{\alpha}+1}v_{\ell,\lambda_{0}}\big)=0.$$
Now consider the remaining elements $x_{-\alpha+r\delta}$, $\alpha\in R^+_s$ respectively $x_{1/2(-\alpha+(2r+1)\delta)}$, $\alpha\in R^+_{\ell}$. 
Recall from our explicit realization of $\CG$ as a subspace of $\mathcal{L}(\overline{\lie g})$ (see Section~\ref{ev0}) that these elements are of the form 
$$X_{\beta_1}^{-}\otimes t^{j}+(-1)^{c+j}X_{\beta_2}^{-}\otimes t^{j}, \mbox{ for some $c\in\bz_+$, $\beta_1,\beta_2\in R_{\overline{\lie g}}^+$},$$
where almost always $j=r$ except for $\alpha=\sum_{q=i}^j\alpha_q+2\sum_{q=j+1}^{n-1}\alpha_q+\alpha_n$, $1\leq i\leq j <n$ we have $j=r-1$. For simplicity we assume $j=r$, since the same argument works for $j=r-1$.
We have
\begin{align*} &\Big(\big(X_{\beta_1}^-+(-1)^{c+r}X_{\beta_2}^-\big)\otimes t^{r}\Big)\mathbf v=\sum^{p}_{q=1}v_{\ell,\ell\overline{\lambda}_{1}}*\cdots*\Big(X_{\beta_1}^- \otimes \big((t+z_q)^{r}-\sum^{r-1}_{k=0} c_k\varphi_{z_q}(t^{k})\big)+&\\&
+X_{\beta_2}^- \otimes \big((-1)^{c+l}(t+z_q)^{r}-\sum^{r-1}_{k=0} (-1)^{c+k}c_k\varphi_{z_q}(t^{k})\big)\Big)v_{\ell,\ell\overline{\lambda}_{q}}*\cdots*v_{\ell,\lambda_{0}}
&\\&+\Big(v_{\ell,\ell\overline{\lambda}_{1}}*\cdots*v_{\ell,\ell\overline{\lambda}_{p}}*\Big(\big(X_{\beta_1}^-+(-1)^{c+r}X_{\beta_2}^-\big)\otimes t^{r}-\sum^{r-1}_{k=0}c_k(X_{\beta_1}^-+(-1)^{c+k}X_{\beta_2}^-\big)\otimes t^{k}\Big)v_{\ell,\lambda_{0}}\Big)
\end{align*}
Again we shall make a particular choice of the coefficients such that for all $1\leq q\leq p$ 
\begin{flalign}\label{diezweite}\notag\Big(X_{\beta_1}^- \otimes \big((t+z_q)^{r}-\sum^{r-1}_{k=0} c_k\varphi_{z_q}(t^{k})\big)
&+X_{\beta_2}^- \otimes \big((-1)^{c+r}(t+z_q)^{r}-\sum^{r-1}_{k=0} (-1)^{c+k}c_k\varphi_{z_q}(t^{k})\big)\Big)v_{\ell,\ell\overline{\lambda}_{q}}&\\& 
=\Big(X_{\beta_1}^-\otimes t^{\overline{\lambda}_q(H_{\beta_1})}+X_{\beta_2}^-\otimes t^{\overline{\lambda}_q(H_{\beta_2})}\Big)v_{\ell,\ell\overline{\lambda}_{q}}\end{flalign}

Write $\lambda(\alpha^{\vee})$ and $\lambda_{0}(\alpha^{\vee})$ as in \eqref{write} and set $r=s_{\alpha}$ and $c_{0},\dots,c_{d-1}=0$. 
We can transform \eqref{diezweite} again into a system of linear equations ($(r-d)$ equations with $(r-d)$ unknowns), where the square matrix $A$ is given as in \eqref{m2} with $r_q=\overline{\lambda}_q(H_{\beta_1}), r_{p+q}=\overline{\lambda}_q(H_{\beta_2})$. Since $A$ is invertible by Lemma~\ref{invmat} we obtain the desired property.
Now suppose $m_{\alpha}<l$, and set $r=s_{\alpha}-1$ and $c_0,\dots,c_{d-2}=0$. Then we can transform \eqref{wieesop} again into a system of linear equations ($(r-d+1)$ equations with $(r-d+1)$ unknowns), where the corresponding square matrix is of similar form.
Hence \eqref{diezweite} holds and since 
$$\big[x_{-\frac{\alpha}{2}+(d-\frac{1}{2})\delta},x_{-\frac{\alpha}{2}+(k+\frac{1}{2})\delta}\big]v_{\ell,\lambda_{0}}=0\mbox{ for all $k\geq d-1$}$$
we can repeat this argument and obtain
$$x_{-\frac{\alpha}{2}+(s_{\alpha}-\frac{1}{2})\delta}^{(2m_{\alpha}-\ell)_{+}+1}\mathbf v=c_{d-1}^{(2m_{\alpha}-\ell)_{+}+1}\big(v_{\ell,\ell\overline{\lambda}_{1}}*\cdots*v_{\ell,\ell\overline{\lambda}_{p}}*x_{-\frac{\alpha}{2}+(d-\frac{1}{2})\delta}^{(2m_{\alpha}-\ell)_{+}+1}v_{\ell,\lambda_{0}}\big)=0$$
$$\mbox{ resp. }x_{-\alpha+(s_{\alpha}-1)\delta}^{m_{\alpha}+1}\mathbf v=c_{d-1}^{m_{\alpha}+1}\big(v_{\ell,\ell\overline{\lambda}_{1}}*\cdots*v_{\ell,\ell\overline{\lambda}_{p}}*x_{-\alpha+(d-1)\delta}^{m_{\alpha}+1}v_{\ell,\lambda_{0}}\big)=0$$
Thus the theorem is proven for the hyperspecial twisted current algebra.

We continue the proof for the remaining twisted current algebras. Let $r\in \bz_+, \alpha\in R_s^+$ and $\overline{\alpha}\in R^{+}_{\overline{\lie g}}$ such that $\overline{\alpha}\vert \lie h=\alpha$. Then,
\begin{align*} x_{-\alpha+r\delta}\mathbf v=&=\Big(X^{-}_{\alpha,r}\otimes t^{ r}\Big)\mathbf v=\Big(X^{-}_{\alpha,r} \otimes t^{r}-\sum^{r-1}_{k=0} c_kX^{-}_{\alpha,k}\otimes t^{k}\Big)\mathbf v&\\&
=\sum^p_{q=1}v_{\ell,\ell\overline{\lambda}_{1}}*\cdots*\Big(X^{-}_{\alpha,r}\otimes (t+z_q)^{r}-\sum^{r-1}_{k=0} c_kX^{-}_{\alpha,k}\otimes(t+z_q)^{k}\Big)v_{\ell,\ell\overline{\lambda}_{q}}*\cdots*v_{\ell,\ell\overline{\lambda}_{p}}
&\\&+ \Big(v_{\ell,\ell\overline{\lambda}_{1}}*\cdots*v_{\ell,\ell\overline{\lambda}_{p}}*\big(x_{-\alpha+r\delta}-\sum_{k=0}^{r-1}c_kx_{-\alpha+k\delta}\big)v_{\ell,\lambda_{0}}\Big)\end{align*}
Again we shall make a particular choice of the coefficients $c_0,\dots,c_{r-1}$ such that for all $1\leq q\leq p$
\begin{equation}\label{hwis}\Big(X^{-}_{\alpha,r}\otimes (t+z_q)^{r}-\sum^{r-1}_{k=0} c_kX^{-}_{\alpha,k}\otimes(t+z_q)^{k}\big)v_{\ell,\ell\overline{\lambda}_{q}}=\Big(\sum^{m-1}_{i=0}(\xi^i)^rX^{-}_{\sigma^i(\overline{\alpha})}\otimes t^{\overline{\lambda}_q(H_{\sigma^i(\overline{\alpha})})}\Big)v_{\ell,\ell\overline{\lambda}_{q}}\end{equation}
Write once more $\lambda(\alpha^{\vee})$ and $\lambda_{0}(\alpha^{\vee})$ as in \eqref{write} and set $r=s_{\alpha}$ and $c_{0},\dots,c_{d-1}=0$.
We can transform \eqref{hwis} once again into a system of linear equations, where the square matrix $A=(a_{i,j}(d))$ is given as in \eqref{m2}. With Lemma~\ref{invmat} we obtain the desired property.
Now suppose $m_{\alpha}<l$. Again setting $r=s_{\alpha}-1$ and $c_{0},\dots,c_{d-2}=0$ we obtain once more with Lemma~\ref{invmat} that 
$$ x_{-\alpha+r\delta}\mathbf v=\Big(v_{\ell,\ell\overline{\lambda}_{1}}*\cdots*v_{\ell,\ell\overline{\lambda}_{p}}*\big(x_{-\alpha+r\delta}-\sum_{k=0}^{r-1}c_kx_{-\alpha+k\delta}\big)v_{\ell,\lambda_{0}}\Big)$$
Therefore,
$$x_{-\alpha+(s_{\alpha}-1)\delta}^{m_{\alpha}+1}\mathbf v=c_{d-1}^{m_{\alpha}+1}\big(v_{\ell,\ell\overline{\lambda}_{1}}*\cdots*v_{\ell,\ell\overline{\lambda}_{p}}*x_{-\alpha+(d-1)\delta}^{m_{\alpha}+1}v_{\ell,\lambda_{0}}\big)=0.$$
The proof 
for all $\alpha\in R^+_{\ell}$ is identical and we omit the details.
\end{pf}
\end{thm}
\begin{cor}
Let $\lambda\in P^+$ such that $\ell\Lambda_0+\lambda\in \widehat{P}^+$. Then we have an isomorphism of $\CG$--modules $$\D(\ell,N\ell\theta+\lambda)\cong_{\CG} \D_{\overline{\lie g}}(\ell,\ell\theta)*\cdots *\D_{\overline{\lie g}}(\ell,\ell\theta)*\ev_0^{*}V(\lambda)$$ 
\end{cor}
We have two remarks.
\begin{rem}\label{danach}
Let $\ell,k\in \mathbb{N}$,\ $\ell k=\ell_1+\cdots+\ell_k$ and $\overline{\lambda}\in \lambda^{\sigma}$.
The same arguments of the proof of Theorem~\ref{mainthmsection6} show that we have a surjective map of $\CG$--modules
$$\D(\ell,\ell k\lambda)\to \D_{\overline{\lie g}}(\ell_1,\ell_1\overline{\lambda})*\cdots* \D_{\overline{\lie g}}(\ell_k,\ell_k\overline{\lambda})\to 0$$
\end{rem}
\begin{rem}
For untwisted modules, the fusion product can be understood as a graded version of the tensor product. From the construction it is well known that the fusion product considered as a module for the underlying simple Lie algebra is isomorphic to the tensor product. Here we have the same circumstances, namely that
$$\D_{\overline{\lie g}}(\ell,\ell\overline{\lambda}_1)*\cdots* \D_{\overline{\lie g}}(\ell,\ell\overline{\lambda}_{p})*\D(\ell,\lambda_{0})\cong _{\lie g}\D(\ell,\ell\lambda_1)\otimes\cdots\otimes \D(\ell,\ell\lambda_{p})\otimes \D(\ell,\lambda_{0})$$
which justifies to call this the fusion product. 
\end{rem}

\subsection{}
For the rest of this section we discuss applications of our result. We begin by noting the following corollary which gives a criterion whether two (non-isomorphic) $\overline{\lie g}[t]$--Demazure modules of same level became isomorphic as $\CG$--modules. 
\begin{prop}
The fusion product of Demazure modules is independent of the choice of the parameters and 
$$\gr \D_{\overline{\lie g}}(\ell, \ell \overline{\lambda}_1)\cong_{\CG} \gr \D_{\overline{\lie g}}(\ell, \ell \overline{\lambda}_2) \mbox{ iff\ \ $\exists \ \lambda\in P^+$ such that } \overline{\lambda}_1,\overline{\lambda}_2\in \lambda^{\sigma}.$$\hfill\qedsymbol
\end{prop}
Another point of view of the previous proposition is the following:
The twisted Demazure module $\D(\ell,\ell \lambda)$ can be obtained by taking the associated graded module of the untwisted Demazure module $\D_{\overline{\lie g}}(\ell, \ell \overline{\lambda})$ for any $\overline{\lambda}\in \lambda^{\sigma}$. As a consequence we obtain together with \cite{FoL06} certain branching rules. For fixed $i$ we set $\epsilon\in\{0,1\}$ such that $\epsilon\equiv i \mod 2$.
\begin{cor}
For $1\leq i\leq n$ we have
\begin{align*}
&V_{\mathfrak{sl}_{2n+1}}(\ell\omega_i)&&\cong_{\mathfrak{sp}_{2n}}V_{\mathfrak{sl}_{2n+1}}(\ell\omega_{2n+1-i})&&\cong_{\mathfrak{sp}_{2n}}\bigoplus_{s_1+\cdots+s_i\leq \ell}V(s_1\omega_1+\cdots+s_i\omega_i)\\
&V_{\mathfrak{sl}_{2n}}(\ell\omega_i)&&\cong_{\mathfrak{sp}_{2n}}V_{\mathfrak{sl}_{2n}}(\ell\omega_{2n-i})&&\cong_{\mathfrak{sp}_{2n}}\bigoplus_{s_\epsilon+\cdots+s_i=\ell}V(s_{\epsilon}\omega_{\epsilon}+s_{\epsilon+2}\omega_{\epsilon+2}+\cdots+s_i\omega_i)\\
&V_{\mathfrak{so}_{2(n+1)}}(\ell\omega_1)&&\cong_{\mathfrak{so}_{2n+1}}\bigoplus_{s\leq \ell}V(s\omega_1)\\
&V_{\mathfrak{so}_{2(n+1)}}(\ell\omega_n)&&\cong_{\mathfrak{so}_{2n+1}}V_{\mathfrak{so}_{2(n+1)}}(\ell\omega_{n+1})&&\cong_{\mathfrak{so}_{2n+1}}V(\ell\omega_n)\\
&V_{E_6}(\ell \omega_1)&&\cong_{F_4}V_{E_6}(\ell \omega_6)&&\cong_{F_4}\bigoplus_{s\leq\ell}V(s\omega_1)\end{align*}
\end{cor}
Even in the untwisted case very little is known about the fusion product of two finite--dimensional irreducible modules. We use our main theorem to cover certain cases for the twisted algebras. We set $\epsilon=1$ if $\CG$ is special and $\epsilon=2$ otherwise.
\begin{cor}
Fix $1\leq i\leq n$ such that $\omega_i$ is a minuscule $\overline{\lie g}$--weight and $\lambda\in P^+$ such that $\epsilon\lambda(\theta^{\vee})\leq \ell$. Then the module $V_{\overline{\lie g}}(\ell \omega_i)^{*k}*\ev^{*}_0V(\lambda)$ is the quotient of $W(\ell k\omega_i+\lambda)$ by the submodule generated by the elements

$$\big\{ (x_{-\alpha+(k\omega_i(\alpha^{\vee})+1)\delta})w_{\ell k\omega_i+\lambda}
:\alpha\in R^+\big\}
\bigcup
\big\{(x_{-\alpha+(k\omega_i(\alpha^{\vee}))\delta})^{\lambda(\alpha^{\vee})+1}w_{\ell k\omega_i+\lambda}:\alpha\in R^+\big\}
$$$$
\bigcup\big\{ (x_{-\frac{\alpha}{2}+(k\omega_i(\alpha^{\vee})+\frac{1}{2})\delta})w_{\ell k\omega_i+\lambda}
:\alpha\in R_{\ell}^+,\ \CG \mbox{ hyperspecial}\big\}$$
\end{cor}
\subsection{}
In what follows we give another application of our main theorem. Fix a non--zero dominant weight $\lambda$ of $\lie g$ and $\ell \in \mathbb{N}$ such that $\Lambda=\ell\Lambda_0+\lambda\in \widehat{P}^+$. 
We shall give a semi--infinite fusion product construction for the irreducible highest weight $\widehat{\lie g}$--module $\widehat{V}(\Lambda)$. 
The special case $\lambda=0$ was proved earlier in \cite{FoL07} and a generalization for untwisted affine Lie algebras was considered in \cite{V13}. Following is the statement of semi--infinite fusion product construction of irreducible representations for the twisted affine algebras:
\begin{thm}\label{semiinfconjecture} 
Let $u\neq0$ be a $\CG$--invariant vector of $\D(\ell, \ell\theta)$ and $\bold{V}^\infty_{\ell, \lambda}$ be the direct limit of $$\ev_0^{*}V(\lambda)\hookrightarrow \D_{\overline{\lie g}}(\ell, \ell\theta)*\ev_0^{*}V(\lambda)\hookrightarrow \D_{\overline{\lie g}}(\ell, \ell\theta)*\D_{\overline{\lie g}}(\ell, \ell\theta)*\ev_0^{*}V(\lambda)\hookrightarrow \cdots $$
where the inclusions are given by $v\mapsto u\otimes v.$ Then
$\widehat{V}(\Lambda) \ \text{and} \ \bold{V}^\infty_{\ell, \lambda}$
are isomorphic as $\CG$--modules.

\end{thm} 
\begin{proof}
Here we follow the ideas of \cite{FoL07}.
By Theorem \ref{mainthmsection6} we have an isomorphism of $\CG$--modules
$$ \D(\ell, (N+1)\ell\theta+\lambda) \cong \D_{\overline{\lie g}}(\ell, \ell \theta) \ast \D(\ell, N\ell\theta+\lambda).$$ Using this isomorphism of Demazure modules, the assertion can be proved in exactly the same way as 
\cite[Theorem 9]{FoL07}. 
\end{proof}

\section{Twisted Q--systems}\label{section7}
In this section, we discuss further consequences of our study and establish the connections
with the twisted $Q$--systems introduced in \cite{HKOTT02}. $Q$--systems for untwisted types have been introduced in \cite{HKOTY99}. We shall use freely and without comment, the notation established in the earlier sections.
\subsection{}
We recall only the definition of the twisted $Q$--system  given in \cite[Section 6]{HKOTT02}. For the untwisted types we refer to \cite[Section 7]{HKOTY99}. Consider the ring $\bz[x_1^{\pm 1},\dots, x_n^{\pm 1}]$  in the indeterminates $x_1,\dots,x_n$, where we recall that $n$ is the rank of $\lie g$. Note that for any $\lambda\in P^+$, the character of $V(\lambda)$ can be regarded as  an element of this ring. A $Q$--system for $\widehat{\lie g}$ is a set of infinitely many commutative
variables $\{Q^{(i)}_j: 1\leq i \leq n, j\in\bz_+\}$  satisfying  $Q^{(i)}_0=1$, and 
\begin{equation}\label{system}Q_j^{(i)}Q_j^{(i)}= Q_{j+1}^{(i)}Q_{j-1}^{(i)}+ \prod_{p\in \Theta(i)} Q_j^{(p)},\end{equation}
where $\Theta(i)$ depends on the type of $\widehat{\lie g}$.
Here, we understand $Q_j^{(i)}=1$ whenever $i\notin\{1,\dots,n\}.$ Below we have listed the choices of $\Theta(i)$.

\vskip 6pt

\noindent 
$\tt A_{2n-1}^{(2)}:$
\begin{eqnarray*}
\Theta(i)&=&\{i-1,i+1\}
\quad \text{for} \quad 1 \le i \le n-1 , \\
\Theta(n)&=&\{n-1,n-1\}
\end{eqnarray*}
\noindent
$\tt A_{2n}^{(2)}:$
\begin{eqnarray*}
\Theta(i)&=&\{ i-1,i+1\}
\quad \text{for} \quad 1 \le i \le n-1 , \\
\Theta(n)&=&\{n-1,n\}. 
\end{eqnarray*}
\noindent
$\tt D_{n+1}^{(2)}:$
\begin{eqnarray*}
\Theta(i)&=&\{i-1,i+1\}
\quad \text{for} \quad 1 \le a \le n-2, \\
 \Theta(n-1)&=&\{n-2,n,n\}, 
\\ 
\Theta(n)&=&\{n-1\}
\end{eqnarray*}
\noindent
$\tt E_{6}^{(2)}:$
\begin{eqnarray*}
\Theta(1)&=&\{2\} , \\
\Theta(2)&=&\{1,3\}, 
\\ 
\Theta(3)&=&\{2,2,4\}, \\ 
\Theta(4)&=& \{3\}.
\end{eqnarray*}
\noindent 
$\tt D_{4}^{(3)}:$
\begin{eqnarray*}
\Theta(1)&=& \{2\}, \\
\Theta(2)&=&\{1,1,1\}.
\end{eqnarray*}
\begin{rem}\label{remqsy}
Let $k$ be maximal such that $k\alpha_i+\alpha_p\in R^+$ for some $1\leq i\leq n$ (resp. $1\leq i< n$ if $\CG$ is hyperspecial) and $p\in \Theta(i)$, then $k d_i\leq \mult_p(\Theta(i))d_p$ where $\mult_p(\Theta(i))$ denotes the multiplicity of $p$ in $\Theta(i)$.
\end{rem}
\subsection{}
Let $\CG$ be a special twisted current algebra and suppose we are given $1\leq i\leq n$ and $\ell\in\bz_+$. The Kirillov--Reshetikhin module $\KRR^{\sigma}(\ell\omega_i)$ defined in \cite[Section 3]{CM06} respectively \cite[Section 2]{CM07} satisfies the simplified defining relations of $V(\bxi(\ell,\ell\omega_i))\cong \D(\ell,\ell\omega_i)$ given in Corollary~\ref{restcorf} and vice versa. It follows that the $\KRR$ module $\KRR^{\sigma}(\ell\omega_i)$ is isomorphic to Demazure module $\D(\ell,\ell\omega_i)$. As far as we know, this fact is nowhere written in the literature for the twisted cases and so we decided to state this result in this paper. Some isomorphisms between $\KRR$ modules and Demazure modules as $\lie g$--modules can be deduced by combining the results of \cite[Section 4]{FoL06} and \cite[Section 3]{CM06} respectively \cite[Section 2]{CM07}. We remark that the isomorphism between the Kirillov--Reshetikhin modules and the Demazure modules for untwisted types was proved earlier in \cite[Section 5]{CM06} and \cite[Section 3.2]{FoL07}.

\begin{prop}\label{prop23}
\label{krdem} Let $\CG$ be a special twisted current algebra. For $1\leq i\leq n$ and $\ell\in\bz_+$, we have an isomorphism of $\CG$--modules, 
\begin{gather*}\KRR^{\sigma}(\ell\omega_i)\cong V(\bxi(\ell, \ell\omega_i))\cong \D(\ell, \ell\omega_i).\end{gather*}\hfill\qedsymbol
\end{prop}

\subsection{}

From now on we suppose again that $\CG$ is hyperspecial or special. Theorem 4.1 of \cite{H10} and Proposition~\ref{prop23} together, prove the following for the special twisted current algebras and 
Theorem 2 of \cite{FoL06} and Theorem 6.3 of \cite{HKOTT02} prove the following for the hyperspecial twisted current algebras.
\begin{prop}\label{qsysteml} 
 The $\lie g$--characters of $\D(\ell,\ell\omega_i)$ satisfy the $Q$--system.  More precisely, for $1\leq i\leq n$ and $\ell\in\bz_+$, we have a (non--canonical) short exact sequence of $\lie g$--modules, 
$$0\to K_{i,\ell}\to \D(\ell,\ell\omega_i)\otimes \D(\ell,\ell\omega_i)\to \D(\ell+1,(\ell+1)\omega_i)\otimes \D(\ell-1,(\ell-1)\omega_i)\to 0,$$ where 
$$K_{i,\ell}\cong \bigotimes_{p\in \Theta(i)}\D(\ell,\ell\omega_p)$$ \hfill\qedsymbol
\end{prop}
\begin{rem}\label{remunt}
It is proved in \cite{H06} that the character of a Kirillov--Reshetikhin module for untwisted types solves the untwisted $Q$--system and hence by our earlier comments likewise the character of $\D_{\overline{\lie g}}(\ell,\ell\omega_i)$, $i\in I$.
\end{rem}
\subsection{}
We shall prove a stronger statement
\begin{thm}\label{qsystemst}
Given $1\leq i\leq n$, $\ell\in\bz_+$ and $\overline{\omega}_j\in \omega_j^{\sigma}$ for $j\in \Theta(i)\cup \{i\}$, we have a (non--canonical) short exact sequence of $\CG$--modules, 
$$0\to K^{*}_{i,\ell}\stackrel{\iota}{\to}  \D(\ell,2\ell\omega_i)\stackrel{\pi}{\to}  \D_{\overline{\lie g}}(\ell+1,(\ell+1)\overline{\omega}_i) * \D(\ell-1,(\ell-1)\omega_i)\to 0$$ 
$$K^{*}_{i,\ell}\cong \scalebox{2}{$\ast$}_{p\in \Theta(i)}\D_{\overline{\lie g}}(\ell,\ell\overline{\omega}_p). $$ 

\end{thm}
The proof of the theorem occupies the rest of this section.
\subsection{}

The following Lemma proves the existence of $\pi$.

\begin{lem}\label{piex} There exists a surjective map of $\CG$--modules  

$$\pi: \D(\ell,2\ell\omega_i)\to \D_{\overline{\lie g}}(\ell+1,(\ell+1)\overline{\omega}_i) * \D(\ell-1,(\ell-1)\omega_i),$$ such that 
 $$0\ne (x_{-\beta})^{\ell}v_{\ell,2\ell\omega_i}\in\ker{\pi},$$
where $\pm\beta=\pm\alpha_i+d_{i}\delta$ if $\CG$ is special or $i=n$ and otherwise $\pm\beta=\frac{1}{2}(\pm\alpha_n+3\delta)$.
\begin{pf}
Recall that the defining relations of Demazure modules are tight and thus by Theorem~\ref{demq} and \cite[Corollary 4.9]{FK11}
$$(x_{-\beta})^{\ell}v_{\ell,2\ell\omega_i}\ne 0.$$ It remains to prove the existence of $\pi$ and  $(x_{-\beta})^{\ell}v_{\ell,2\ell\omega_i}\in\ker{\pi}$. We give the proof only for the hyperspecial twisted current algebra, since the other cases proceed similarly. 
Since 
$$\big(X_{i,i}^-\otimes t\big)v_{\ell+1,(\ell+1)\overline{\omega}_i}=\big(X_{2n+1-i,2n+1-i}^-\otimes t\big)v_{\ell+1,(\ell+1)\overline{\omega}_i}=0$$ 
and depending on the choice of $\overline{\omega}_i$
$$\big(X_{i,i}^-\otimes 1\big)v_{\ell+1,(\ell+1)\overline{\omega}_i}=0,\ \mbox{ or }\ \big(X_{2n+1-i,2n+1-i}^-\otimes 1\big)v_{\ell+1,(\ell+1)\overline{\omega}_i}=0$$ 
we can choose similar to \eqref{diezweite} a complex number $c\in \bc$ such that for $r=1,2$
$$\big(x_{-\alpha_i+rd_{i}\delta}-cx_{-\alpha_i+(r-1)\delta}\big)v_{\ell+1,(\ell+1)\overline{\omega}_i}=0,\ \mbox{ resp. }\big(x_{\frac{1}{2}(-\alpha_n+(2r+1)\delta)}-cx_{\frac{1}{2}(-\alpha_n+(2r-1)\delta)}\big)v_{\ell+1,(\ell+1)\overline{\omega}_i}=0$$
Thus $r=2$ ensures the existence of $\pi$ and $r=1$ shows that $x_{-\beta}$ acts only on the second factor as $x_{-\alpha_i}$ respectively $x_{\frac{1}{2}(-\alpha_n+\delta)}$. Therefore,
$$(x_{-\beta})^{\ell}v_{\ell+1,(\ell+1)\overline{\omega}_i}*v_{\ell-1,(\ell-1)\omega_i}=v_{\ell+1,(\ell+1)\overline{\omega}_i}*(x_{-\alpha_i})^{\ell}v_{\ell-1,(\ell-1)\omega_i}=0$$
$$\mbox{resp. }(x_{-\beta})^{\ell}v_{\ell+1,(\ell+1)\overline{\omega}_i}*v_{\ell-1,(\ell-1)\omega_i}=v_{\ell+1,(\ell+1)\overline{\omega}_i}*(x_{\frac{1}{2}(-\alpha_n+\delta)})^{\ell}v_{\ell-1,(\ell-1)\omega_i}=0.$$
\end{pf}
\end{lem}
\subsection{}
The next result establishes the existence of $\iota$. 
\begin{lem}\label{iotax} There exists an injective non-zero map of $\CG$--modules, 
$$\widetilde{\iota}:K_{i,\ell}^*\to \ker\pi.$$
  \end{lem}
  \begin{pf} 
Let $w\in \widehat{W}$ and $\Lambda$ be the dominant integral $\widehat{\lie g}$--weight such that $\D(\ell,2\ell\omega_i)$ is the $\widehat{\lie b}$-module generated by the line $\widehat{V}(\Lambda)_{\omega(\Lambda)}$. Hence $w(\Lambda)=\ell \Lambda_0-2\ell\omega_i+r\delta$ for some $r\in \bz$. Recall that $\D(\ell,2\ell\omega_i)$ is $\CG$-stable and is generated as a $\CG$--module by the line $\widehat{V}(\Lambda)_{w_0\omega(\Lambda)}$. We obtain 
$$w_{\alpha_i+d_i\delta}\omega(\Lambda)=\ell\Lambda_0-\ell\big(\sum_{p\in\theta(i)}\omega_p\big)+(r+\ell d_i)\delta,$$ $$\big(\mbox{resp. }w_{\frac{1}{2}(\alpha_n+3\delta)}\omega(\Lambda)=\ell\Lambda_0-\ell(\omega_{n-1}+\omega_n)+(r+\frac{3}{2}\ell)\delta.\big) $$
Hence the $\widehat{\lie b}$-module generated by the line $\widehat{V}(\Lambda)_{w_{\beta}\omega(\Lambda)}$ is $\CG$--stable and is generated as a $\CG$--module by the line  $\widehat{V}(\Lambda)_{w_0w_{\beta}\omega(\Lambda)}=\bc (x_{-\beta})^{\ell}v_{\ell,2\ell\omega_i}$. Therefore,
$$K_{i,\ell}^*\cong \D\big(\ell, \ell\big(\sum_{p\in\theta(i)}\omega_p\big)\big)\cong \bu(\CG)\widehat{V}(\Lambda)_{w_0w_{\beta}\omega(\Lambda)}\cong \bu(\CG)(x_{-\beta})^{\ell}v_{\ell,2\ell\omega_i}\hookrightarrow \ker \pi.$$

\end{pf}
\subsection{}
By Lemma~\ref{piex} we have a short exact sequence 
$$0\to \ker \pi{\to}  \D(\ell,2\ell\omega_i)\stackrel{\pi}{\to}  \D_{\overline{\lie g}}(\ell+1,(\ell+1)\overline{\omega}_i) * \D(\ell-1,(\ell-1)\omega_i)\to 0,$$
which is non-split, since $\D(\ell,2\ell\omega_i)$ is indecomposable.
Together with Lemma~\ref{iotax} and Proposition~\ref{qsysteml} we obtain
\begin{align*}\dim \D(\ell,2\ell\omega_i)&= \dim (\ker \pi) + \dim\big(\D_{\overline{\lie g}}(\ell+1,(\ell+1)\overline{\omega}_i) * \D(\ell-1,(\ell-1)\omega_i)\big)&\\&
\geq \dim(K_{i,\ell}^{*})+\dim\big(\D_{\overline{\lie g}}(\ell+1,(\ell+1)\overline{\omega}_i) * \D(\ell-1,(\ell-1)\omega_i)\big)&\\&
\geq \dim(K_{i,\ell})+\dim\big(\D(\ell+1,(\ell+1)\omega_i) \otimes  \D(\ell-1,(\ell-1)\omega_i)\big)&\\&
=\dim \D(\ell,2\ell\omega_i).
\end{align*}
Hence $\widetilde{\iota}$ is an isomorphism proving Theorem~\ref{qsystemst}.
We complete the paper by giving a proof of Lemma~\ref{lemwe}.

\subsection{Proof of Lemma~\ref{lemwe}}
By the tensor product decomposition of untwisted Demazure modules proved in \cite{FoL06} and Theorem~\ref{tensordec} it suffices to show $\dim_{\overline{\lie g}}\D(\ell,\ell \omega_i)=\dim \D(\ell,\ell \omega_i)$ for $1\leq i \leq n$. We prove this equality by induction on $\ell$. If $\ell=1$ this follows from \cite[Lemma 5.3]{FK11} if $\CG$ is special and otherwise
$$\dim \D_{\overline{\lie g}}(1,\omega_i)=\dim V_{\overline{\lie g}}(\omega_i)=1+\sum^i_{j=1}V(\omega_j)=\dim \D(1,\omega_i),\ \mbox{ for $1\leq i \leq n$}$$
and the induction begins. 
Since the characters of $\D_{\overline{\lie g}}(\ell,\ell\omega_i)$ and $\D(\ell,\ell \omega_i)$ respectively solve the untwisted and twisted $Q$--system respectively we can deduce from Proposition~\ref{qsysteml} and Remark~\ref{remunt} and the induction hypothesis
$$\dim \D_{\overline{\lie g}}((\ell+1),(\ell+1)\omega_i)\dim \D((\ell-1),(\ell-1)\omega_i)=\dim \D(\ell,2\ell\omega_i)-\prod_{p\sim i}\dim \D_{\overline{\lie g}}(\ell,\ell \omega_p),$$
$$\dim \D((\ell+1),(\ell+1)\omega_i)\dim \D((\ell-1),(\ell-1)\omega_i)=\dim \D(\ell,2\ell\omega_i)-\prod_{p\in\sigma(i)}\dim \D(\ell,\ell \omega_p),$$
where $p \sim i$ means $(\alpha_p,\alpha_i)<0$. 
Now consider the set $\{p\mid p \sim i\}$ and replace any $p>n$ by $\sigma^j(p)$ where $j$ is minimal such that $\sigma^j(p)\leq n$. The corresponding set is equal to $\theta(i)$. For example if $\overline{\lie g}=\tt A_{2n}$ and $i=n$ we have $\{p\mid p \sim i\}=\{n-1,n+1\}$ and we replace $n+1$ by $n$. The Lemma follows now easily.





\bibliographystyle{plain}
\bibliography{kv-bib}
\end{document}